\documentclass[reqno,10pt,centertags,draft]{amsart}
\usepackage{amsmath,amsthm,amscd,amssymb,latexsym,upref}
\date{\today}

\newcommand{\bbN}{{\mathbb{N}}}
\newcommand{\bbR}{{\mathbb{R}}}

\newcommand{\bbC}{{\mathbb{C}}}

\newcommand{\cB}{{\mathcal B}}

\newcommand{\cH}{{\mathcal H}}
\newcommand{\cI}{{\mathcal I}}

\newcommand{\cM}{{\mathcal M}}

\newcommand{\cR}{{\mathcal R}}

\newcommand{\gR}{\mathfrak{R}}
\newcommand{\gt}{\mathfrak{t}}

\newcommand{\no}{\notag}
\newcommand{\lb}{\label}
\newcommand{\f}{\frac}

\newcommand{\ol}{\overline}

\newcommand{\wti}{\widetilde}

\newcommand{\ran}{\text{\rm{ran}}}

\newcommand{\dom}{\text{\rm{dom}}}

\newcommand{\bi}{\bibitem}

\newcommand{\beq}{\begin{equation}}
\newcommand{\eeq}{\end{equation}}
\newcommand{\ba}{\begin{align}}
\newcommand{\ea}{\end{align}}

\renewcommand{\Re}{\text{\rm Re}}
\renewcommand{\Im}{\text{\rm Im}}

\renewcommand{\ge}{\geqslant}
\renewcommand{\le}{\leqslant}

\DeclareMathOperator{\sgn}{sgn}

\DeclareMathOperator*{\slim}{s-lim}

\DeclareMathOperator{\diag}{diag}

\newcommand{\la}{\lambda}

\allowdisplaybreaks \numberwithin{equation}{section}

\newtheorem{theorem}{Theorem}[section]
\newtheorem{lemma}[theorem]{Lemma}
\newtheorem{corollary}[theorem]{Corollary}

\newtheorem{example}[theorem]{Example}
\theoremstyle{definition}
\newtheorem{definition}[theorem]{Definition}
\newtheorem{remark}[theorem]{Remark}

\begin{document}

\title[Generalized Polar Decompositions of Closed Operators in Hilbert Spaces]{Generalized Polar Decompositions for Closed Operators in Hilbert Spaces and Some Applications}
\author[F.\ Gesztesy, M.\ Malamud, M.\ Mitrea, and S.\ Naboko]{Fritz Gesztesy, Mark Malamud, Marius Mitrea, and Serguei Naboko}
\address{Department of Mathematics,
University of Missouri, Columbia, MO 65211, USA}
\email{fritz@math.missouri.edu}
\urladdr{http://www.math.missouri.edu/personnel/faculty/gesztesyf.html}
\address{Mathematics, Institute of Applied Mathematics and Mechanics, R. Luxemburg
str. 74, Donetsk 83114, Ukraine}
\email{mmm@telenet.dn.ua}
\address{Department of Mathematics, University of
Missouri, Columbia, MO 65211, USA}
\email{marius@math.missouri.edu}
\urladdr{http://www.math.missouri.edu/personnel/faculty/mitream.html}
\address{Department of Mathematical Physics, Institute of Physics, St. Petersburg State University,
1 Ulia- novskaia, St. Petergoff, St. Petersburg, 198504, Russia, and Department of Mathematics, University of Alabama at Birmingham, Birmingham, AL 35294-1170, USA}
\email{naboko@snoopy.phys.spbu.ru}
\thanks{Based upon work partially supported by the US National Science
Foundation under Grant Nos.\ DMS-0400639 and FRG-0456306, and
the Austrian Science Fund (FWF) under Grant No.\ Y330.}
\date{\today}

\begin{abstract}
We study generalized polar decompositions of densely defined, closed linear operators in Hilbert spaces and provide some applications to relatively (form) bounded and relatively (form) compact perturbations of self-adjoint, normal, and $m$-sectorial operators.
\end{abstract}

\maketitle

\section{Introduction} \lb{s1}

This paper had its origin in attempts of proving that certain operators of the type
\begin{equation}
 \ol{(A +  I_{\cH})^{-1/2} B (A +  I_{\cH})^{-1/2}},   \lb{1.1}
\end{equation}
in a complex, separable Hilbert space $\cH$ (where $\ol S$ denotes the
closure of the operator $S$ and $I_\cH$ is the identity operator in $\cH$), are bounded, respectively, compact,
where $A\ge 0$ is self-adjoint in $\cH$, and $B$ is a densely defined, closed operator
in $\cH$. To prove such a result, it became desirable to replace the standard polar decomposition of $B$
(cf.\ \cite[Sect.\,IV.3]{EE89}, \cite[Sect.\,VI.2.7]{Ka80}),
\begin{equation}
B=U|B|=|B^*|U \, \text{ on } \, \dom(B) = \dom(|B|),   \lb{1.2}
\end{equation}
by some modified polar decomposition of the type
\begin{equation}
B = |B^*|^{1/2} U |B|^{1/2}  \, \text{ on } \, \dom(B) = \dom(|B|),     \lb{1.3}
\end{equation}
and then reduce boundedness, respectively, compactness of the operator \eqref{1.1}
to that of
\begin{equation}
|B|^{1/2} (A +  I_{\cH})^{-1/2} \, \text{ and } \, |B^*|^{1/2} (A +  I_{\cH})^{-1/2}.   \lb{1.4}
\end{equation}
With \eqref{1.3} in mind, it is natural to try to establish that, in fact, the following version of \eqref{1.3} holds
\begin{equation}
B = |B^*|^\alpha U |B|^{1-\alpha}  \, \text{ on } \, \dom(B) = \dom(|B|)    \lb{1.5}
\end{equation}
for all $\alpha \in [0,1]$.  In fact, after this was accomplished, it became clear that the following rather general polar-type decomposition can be established
\begin{equation}
B = \phi(|B^*|) U \psi(|B|)  \, \text{ on } \, \dom(B) = \dom(|B|),   \lb{1.6}
\end{equation}
where $\phi$ and $\psi$ are Borel functions on $\bbR$ with the property that
$\phi(\lambda) \psi(\lambda) = \lambda$, $\lambda \in \bbR$, and such that
$\dom(|B|) \subseteq \dom(\psi(|B|))$.

Finally, an even more general version of \eqref{1.6} is to show that an operator $T$ introduced as
\begin{equation}
T = VA_1 = A_2 V \, \text{ on } \,  \dom(T)=\dom(A_1),   \lb{1.7}
\end{equation}
also has the representation
\begin{equation}
T = \phi(A_2) V \psi(A_1)  \, \text{ on } \, \dom(T) = \dom(A_1)   \lb{1.8}
\end{equation}
for any pair of self-adjoint (in fact, also normal) operators $A_j$, $j=1,2$, and any bounded operator $V$
satisfying $V \dom(A_1) \subseteq \dom(A_2)$, assuming also $\dom(A_1) \subseteq \dom(\psi(A_1))$
(cf.\ Theorems \ref{t2.1} and \ref{t2.3} for details).

In Section \ref{s2} we provide proofs of \eqref{1.6} and \eqref{1.8}, and in Section \ref{s3} we discuss some applications to relatively (form) bounded and relatively (form) compact perturbations of self-adjoint operators. In the final Section \ref{s4} we discuss some applications to $m$-sectorial operators.

\section{Generalized Polar Decompositions} \lb{s2}

To set the stage, let $\cH_j$, $j=1,2$, be two separable complex Hilbert spaces with scalar
products and norms denoted by $(\cdot,\cdot)_{\cH_j}$ and $\|\cdot \|_{\cH_j}$, $j=1,2$, respectively. The identity operators in $\cH_j$ are written as $I_{\cH_j}$, $j=1,2$.
We denote by $\cB(\cH_1,\cH_2)$ (resp., $\cB_{\infty}(\cH_1,\cH_2)$) the Banach space of linear bounded (resp., compact) operators from $\cH_1$ into $\cH_2$. If $\cH_1=\cH_2=\cH$, these spaces are denoted by $\cB(\cH)$ (resp., $\cB_{\infty}(\cH)$). The domain, range, kernel (null space), resolvent
set, and spectrum  of a linear operator will be denoted by $\dom(\cdot)$, $\ran(\cdot)$, $\ker(\cdot)$,
$\rho(\cdot)$, and $\sigma(\cdot)$, respectively. Finally, we let $\ol S$ stand for the closure of an operator $S$.

We assume that
\begin{equation}
\text{$A_j$ are self-adjoint operators in $\cH_j$ with domains $\dom(A_j)$, $j=1,2$,}    \lb{2.1}
\end{equation}
and that
\begin{equation}
V \in \cB(\cH_1, \cH_2)       \lb{2.2}
\end{equation}
satisfies
\begin{equation}
V \dom(A_1) \subseteq \dom(A_2).   \lb{2.3}
\end{equation}
In addition, suppose that
\begin{equation}
VA_1 = A_2 V \, \text{ on } \, \dom(A_1).   \lb{2.4}
\end{equation}

Next,  given a self-adjoint operator $A$ in a complex separable Hilbert space $\cH$, we  denote by
$\{E_A(\lambda)\}_{\la\in\bbR}$ the family of spectral projections associated with $A$, and we introduce the function $\rho_f $ by
\begin{equation}
\rho_f \colon \begin{cases} \bbR \to [0,\infty), \\
\hspace*{.6mm} \la \mapsto \|E_A(\la) f\|_{\cH}^2, \end{cases}   f\in\cH.   \lb{2.4a}
\end{equation}
Clearly, $\rho_f$ is bounded, non-decreasing, right-continuous, and
\begin{equation}
\lim_{\la\downarrow -\infty} \rho_f(\la)=0, \quad \lim_{\la\uparrow \infty} \rho_f(\la)=\|f\|_{\cH}^2,
\quad f\in\cH.     \lb{2.4b}
\end{equation}
Hence, $\rho_f$ generates a measure, denoted by $d \rho_f$, in a canonical manner.

A function $\phi \colon\bbR\to\bbC$ is then called
{\it $d E_A$-measurable} if it is $d \rho_f$-measurable for all $f\in\cH$. Standard examples of $d E_A$-measurable functions are all continuous functions, all step functions, all pointwise limits of step functions, and all Borel measurable functions. Given a $d E_A$-measurable function $\phi$, the operator $\phi(A)$ is then defined in terms of the spectral representation of $A$ as usual by
\begin{equation}
 \phi(A) = \int_{\bbR} d E_A(\lambda) \, \phi(\lambda), \quad
 \dom(\phi(A)) = \bigg\{f\in\cH \,\bigg|\, \int_{\bbR} d\|E_A(\lambda) f\|_{\cH}^2 \, |\phi(\lambda)|^2 < \infty
\bigg\}.
\end{equation}

Our first result result then reads as follows:

\begin{theorem}   \lb{t2.1}
Suppose $A_j$, $j=1,2$, and $V$ satisfy \eqref{2.1}--\eqref{2.4}, and consider the operator $T$ given by
\begin{equation}
T = VA_1 = A_2 V \, \text{ on } \, \dom(T)=\dom(A_1).   \lb{2.5}
\end{equation}
$(i)$ If $\psi$ is both a $d E_{A_1}$- and $d E_{A_2}$-measurable function on $\bbR$, then
\begin{equation}
V \dom(\psi(A_1)) \subseteq \dom(\psi(A_2))    \lb{2.6}
\end{equation}
and
\begin{equation}
V \psi(A_1) = \psi(A_2) V  \, \text{ on } \, \dom(\psi(A_1)).   \lb{2.7}
\end{equation}
$(ii)$ Assume that $\phi$ and $\psi$ are simultaneously $d E_{A_1}$-
and $d E_{A_2}$-measurable functions on $\bbR$ such that
\begin{equation}
\phi(\lambda) \psi(\lambda) = \lambda, \quad \lambda \in \bbR,    \lb{2.8}
\end{equation}
and
\begin{equation}
\dom(A_1) \subseteq \dom(\psi(A_1)).    \lb{2.9}
\end{equation}
Then
\begin{equation}
T = \phi(A_2) V \psi(A_1)  \, \text{ on } \, \dom(T) = \dom(A_1).   \lb{2.10}
\end{equation}
\end{theorem}
\begin{proof}
Since $VA_1\subseteq A_2 V$, one infers $V (A_1 - z I_{\cH_1}) \subseteq (A_2 - z I_{\cH_2}) V$ and hence
\begin{equation}
V (A_1 - z I_{\cH_1})^{-1} = (A_2 - z I_{\cH_2})^{-1} V, \quad z \in\bbC\backslash\bbR.   \lb{2.11}
\end{equation}
In the following we denote by $\{E_{A_j}(\lambda)\}_{\lambda\in\bbR}$ the family of (strongly right-continuous) spectral projections of the self-adjoint operators $A_j$, $j=1,2$. Then, the representation
(cf.\ \cite[Sect.\ VI.5.2]{Ka80})
\begin{equation}
E_{A_j}(\lambda) = I_{\cH_j}
-\tfrac{1}{2} \big[U_j (\lambda) + U_j (\lambda)^2\big], \quad \lambda \in \bbR, \; j=1,2,    \lb{2.12}
\end{equation}
where
\begin{equation}
U_j (\lambda) = \slim_{\varepsilon\downarrow 0, \, R\uparrow \infty} \f{2}{\pi} \int_\varepsilon^R
d\eta \, (A_j - \lambda I_{\cH_j}) \big[(A_j - \lambda I_{\cH_j})^2 + \eta^2 I_{\cH_j}\big]^{-1},
\quad \lambda \in \bbR,  \; j=1,2,    \lb{2.13}
\end{equation}
(here $\slim$ denotes the strong limit in $\cH_j$) yields
\begin{equation}
V E_{A_1}(\lambda) = E_{A_2}(\lambda) V, \quad \lambda \in \bbR.   \lb{2.14}
\end{equation}

Next, choose $f\in \dom(\psi(A_1))$. Then
\begin{align}
& \int_{-R}^R d \|E_{A_2}(\lambda) V f\|^2_{\cH_2} |\psi(\la)|^2
 = \int_{-R}^R d \|V E_{A_1}(\lambda) f\|^2_{\cH_2} |\psi(\la)|^2  \no \\
& \quad \underset{{R \uparrow \infty} }\longrightarrow
\int_{\bbR} d \|V E_{A_1}(\lambda) f\|^2_{\cH_2} |\psi(\la)|^2
\leq  \|V\|^2_{\cB(\cH_1,\cH_2)} \|\psi (A_1) f\|^2_{\cH_1}.    \lb{2.15}
\end{align}
Thus, $f\in \dom(\psi(A_1))$ implies $V f\in \dom(\psi(A_2))$, proving \eqref{2.6}.

Choosing $f\in \dom(\psi(A_1))$ and $g\in\cH_2$ then yields
\begin{align}
(V \psi(A_1)f,g)_{\cH_2} &= \int_{\bbR} d(V E_{A_1}(\la)f,g)_{\cH_2} \psi(\la)    \no \\
&= \int_{\bbR} d(E_{A_2}(\la) V f,g)_{\cH_2} \psi(\la)    \no \\
&= (\psi(A_2) V f,g)_{\cH_2},    \lb{2.16}
\end{align}
and hence \eqref{2.7} is proven.

Finally, \eqref{2.10} follows from \eqref{2.7}--\eqref{2.9} since
\begin{align}
\phi(A_2) V \psi(A_1) = \phi(A_2) \psi(A_2) V = A_2 V = T,     \lb{2.17}
\end{align}
concluding the proof.
\end{proof}

\begin{remark} \lb{r2.2}
$(i)$ The crucial intertwining relation \eqref{2.14} also follows from
\eqref{2.11} and the Stieltjes inversion formula for (finite) complex
measures (cf., e.g., \cite[App.\ B]{We80}). Indeed,
\begin{equation}
\int_{\bbR} d(V E_{A_1}(\la) f,g)_{\cH_2} (\la -z)^{-1}
= \int_{\bbR} d(E_{A_2}(\la) V f,g)_{\cH_2} (\la -z)^{-1}, \quad z\in\bbC\backslash\bbR, \; f\in\cH_1, \,
g\in\cH_2,      \lb{2.18}
\end{equation}
implies
\begin{equation}
d(V E_{A_1}(\cdot) f,g)_{\cH_2} = d(E_{A_2}(\cdot) V f,g)_{\cH_2}, \quad f\in\cH_1, \, g\in\cH_2,
\lb{2.18a}
\end{equation}
and hence \eqref{2.14}. \\
$(ii)$ In the special case where in addition to \eqref{2.1}--\eqref{2.4}, $A_j$ are bounded,
$A_j\in\cB(\cH_j)$, $j=1,2$, one can also derive \eqref{2.10} for functions $\phi$ and $\psi$ continuous in an open neighborhood  of the spectra of $A_1$ and $A_2$ using a Stone--Weierstrass approximation argument.
\end{remark}

Now we turn our attention to a pair of {\it normal}  operators $A_j$, $j=1,2$, with the aim of proving
the analog of Theorem \ref{t2.1} in this case. For an extensive treatment of normal operators and the spectral family and spectral theorem associated with them, we refer to
\cite[Sects.\ 5.6 and 7.5]{We80}.

Thus, we assume that
\begin{equation}
\text{$A_j$ are normal operators in $\cH_j$ with domains $\dom(A_j)$, $j=1,2$,}    \lb{2.14a}
\end{equation}
(i.e., $A_j A_j^* = A_j^* A_j$ and $\dom(A_j)=\dom(A_j^*)$, $j=1,2$) such that
\begin{equation}
\rho(A_1)\cap\rho(A_2) \neq \emptyset.    \lb{2.14b}
\end{equation}
In addition, suppose that
\begin{equation}
V \in \cB(\cH_1, \cH_2)       \lb{2.14c}
\end{equation}
satisfies
\begin{equation}
V \dom(A_1) \subseteq \dom(A_2).   \lb{2.14d}
\end{equation}
and assume that
\begin{equation}
VA_1 = A_2 V \, \text{ on } \, \dom(A_1).   \lb{2.14e}
\end{equation}

Given a normal operator $A$ in a complex separable Hilbert space $\cH$ we  denote by
$\{E_A(\nu)\}_{\nu\in\bbC}$ the family of spectral projections associated with $A$. We
recall that $\ol{(A + A^*)/2}$ and $\ol{(A  - A^*)/(2 i)}$ are self-adjoint, and we denote by
$\{ E_{\ol{(A + A^*)/2}}(\la)\}_{\la\in\bbR}$ and $\{ E_{\ol{(A - A^*)/(2i)}}(\la)\}_{\la\in\bbR}$ the corresponding family of spectral projections. Then the family of spectral projections
$\{ E_{A}(\nu)\}_{\nu\in\bbC}$ for the normal operator $A$ is given by (cf.\ \cite[Theorem\ 7.32]{We80})
\begin{align}
\begin{split}
& E_{A}(\nu) = E_{\ol{(A + A^*)/2}}(\la) \, E_{\ol{(A - A^*)/(2i)}}(\mu)
= E_{\ol{(A - A^*)/(2i)}}(\mu) \, E_{\ol{(A + A^*)/2}}(\la),    \\
& \hspace*{6.9cm} \nu = \lambda + i \mu \in \bbC, \; \la, \mu \in\bbR.    \lb{2.14eA}
\end{split}
\end{align}

In analogy to the self-adjoint case one then defines the function $\tau_f $ by
\begin{equation}
\tau_f \colon \begin{cases} \bbC \to [0,\infty), \\
\hspace*{.6mm} \nu \mapsto \|E_A(\nu) f\|_{\cH}^2, \end{cases}   f\in\cH.   \lb{2.14eB}
\end{equation}
As discussed in \cite[Appendix\ A.1]{We80},  introducing
\begin{equation}
N = L \times M = \{z\in\bbC \,|\, \Re(z)\in L, \, \Im(z) \in M\}
\end{equation}
for arbitrary intervals $L, M \subseteq \bbR$, then
\begin{equation}
\tau_f(N) = \|E_A(N) f\|_{\cH}^2 = \|E_{\ol{(A + A^*)/2}}(L) \, E_{\ol{(A - A^*)/(2i)}}(M) f\|_{\cH}^2
\end{equation}
defines a regular interval function and hence a measure $d \tau_f$ for each
$f\in\cH$. A function $\phi \colon\bbC\to\bbC$ is then called
{\it $d E_A$-measurable} if it is $d \tau_f$-measurable for all $f\in\cH$.

\begin{theorem}   \lb{t2.3}
Suppose $A_j$, $j=1,2$, and $V$ satisfy \eqref{2.14a}--\eqref{2.14e}, and consider the operator $T$ given by
\begin{equation}
T = VA_1 = A_2 V \, \text{ on } \, \dom(T)=\dom(A_1).   \lb{2.14f}
\end{equation}
$(i)$ If $\psi$ is both a $d E_{A_1}$- and $d E_{A_2}$-measurable function on $\bbC$ then
\begin{equation}
V \dom(\psi(A_1)) \subseteq \dom(\psi(A_2))    \lb{2.14g}
\end{equation}
and
\begin{equation}
V \psi(A_1) = \psi(A_2) V  \, \text{ on } \, \dom(\psi(A_1)).   \lb{2.14h}
\end{equation}
$(ii)$ Assume that $\phi$ and $\psi$ are simultaneously $d E_{A_1}$- and $d E_{A_2}$-measurable functions
on $\bbC$ such that
\begin{equation}
\phi(\lambda) \psi(\lambda) = \lambda, \quad \lambda \in \bbC,    \lb{2.14i}
\end{equation}
and
\begin{equation}
\dom(A_1) \subseteq \dom(\psi(A_1)).    \lb{2.14j}
\end{equation}
Then
\begin{equation}
T = \phi(A_2) V \psi(A_1)  \, \text{ on } \, \dom(T) = \dom(A_1).   \lb{2.14k}
\end{equation}
\end{theorem}
\begin{proof}
The idea of the proof is to try to reduce the case of normal operators to that of self-adjoint ones treated in Theorem \ref{t2.1}. With this goal in mind, pick $z\in \bbC\backslash(\sigma(A_1)\cup\sigma(A_2))$ for the remainder of this proof. Then $VA_1 \subseteq A_2V$ implies again
\begin{equation}
V (A_1 - z I_{\cH_1})^{-1} = (A_2 - z I_{\cH_2})^{-1} V,    \lb{2.14l}
\end{equation}
and hence also
\begin{equation}
Ve^{i {\ol \zeta} (A_1 - z I_{\cH_1})^{-1}} = e^{i {\ol \zeta} (A_2 - z I_{\cH_2})^{-1}} V, \quad
\zeta\in\bbC,   \lb{2.14m}
\end{equation}
applying \eqref{2.14l} repeatedly to all terms in the norm convergent Taylor expansion of both  exponentials in \eqref{2.14m}. (Here $\ol \zeta$ denotes the complex conjugate of $\zeta\in\bbC$.) In particular,
\begin{equation}
V = e^{i {\ol \zeta} (A_2 - z I_{\cH_2})^{-1}} V e^{-i {\ol \zeta} (A_1 - z I_{\cH_1})^{-1}},  \quad
\zeta\in\bbC.   \lb{2.14n}
\end{equation}
Thus, one obtains
\begin{align}
e^{i \zeta (A_2^*-\ol z)^{-1}} V e^{-i \zeta (A_1^*-\ol z)^{-1}} & =
e^{i \zeta (A_2^*-\ol z)^{-1}} e^{i {\ol \zeta} (A_2 - z I_{\cH_2})^{-1}} V
e^{-i {\ol \zeta} (A_1 - z I_{\cH_1})^{-1}} e^{-i \zeta (A_1^*-\ol z)^{-1}} \no \\
& = e^{i B_2(\zeta)} V e^{-i B_1(\zeta)},   \quad \zeta\in\bbC,    \lb{2.14o}
\end{align}
where we have set
\begin{align}
\begin{split}
B_1(\zeta) =  \zeta (A_1^*-\ol z)^{-1} + {\ol \zeta} (A_1 - z I_{\cH_1})^{-1} = B_1(\zeta)^*,   \\
B_2(\zeta) = \zeta (A_2^*-\ol z)^{-1} + {\ol \zeta} (A_2 - z I_{\cH_2})^{-1} = B_2(\zeta)^*.   \lb{2.14p}
\end{split}
\end{align}
Consequently,
\begin{equation}
\big\|e^{i \zeta (A_2^*-\ol z)^{-1}} V e^{-i \zeta (A_1^*-\ol z)^{-1}}\big\|_{\cB(\cH_1,\cH_2)}
= \|V\|_{\cB(\cH_1,\cH_2)},
\quad \zeta \in \bbC.    \lb{2.14q}
\end{equation}
Since the left-hand side of \eqref{2.14o} is entire with respect to $\zeta\in\bbC$, the uniform boundedness in \eqref{2.14q} and Liouville's theorem yield that  the left-hand side of \eqref{2.14o} is actually constant with respect to $\zeta\in\bbC$. Thus, the left-hand side of \eqref{2.14o} equals its value at $\zeta=0$, allowing one to conclude that
\begin{equation}
e^{i \zeta (A_2^*-\ol z)^{-1}} V e^{-i \zeta (A_1^*-\ol z)^{-1}}  = V, \quad \zeta\in\bbC.   \lb{2.14r}
\end{equation}
Differentiating \eqref{2.14r} with respect to $\zeta$ and subsequently taking $\zeta =0$, then yields
\begin{equation}
V (A_1^*-\ol z)^{-1} = (A_2^*-\ol z)^{-1} V,   \lb{2.14s}
\end{equation}
and consequently,
\begin{equation}
V A_1^* = A_2^* V \, \text{ on } \, \dom(A_1^*).   \lb{2.14t}
\end{equation}

Equations \eqref{2.14f} and \eqref{2.14t} together imply
\begin{equation}
V (A_1 \pm A_1^*) = (A_2 \pm A_2^*) V \, \text{ on } \, \dom(A_1)=\dom(A_1^*).   \lb{2.14u}
\end{equation}

Next we will show that \eqref{2.14u} extends to the closures of $A_j \pm A_j^*$, $j=1,2$, as follows: First, we note that $A_j \pm A_j^*$, $j=1,2$, are symmetric and hence closable. Next, pick arbitrary
$f_\pm \in \dom\big(\ol{A_1\pm A_1^*}\big)$ and let
$f_{\pm,n}\in\dom(A_1)= \dom(A_1^*)$ be such that
\begin{equation}
\lim_{n\to\infty} \|f_{\pm,n} - f_\pm \|_{\cH_1} =0 \, \text{ and } \,
\lim_{n\to\infty} \big\|(A_1\pm A_1^*)f_{\pm,n} - (\ol{A_1\pm A_1^*})f_\pm \big\|_{\cH_1} =0.   \lb{2.14v}
\end{equation}
Given that $V\in\cB(\cH_1,\cH_2)$, one also has
\begin{equation}
\lim_{n\to\infty} \|V f_{\pm,n} - V f_\pm \|_{\cH_2} =0 \, \text{ and } \,
\lim_{n\to\infty} \big\|V (A_1\pm A_1^*)f_{\pm,n} - V(\ol{A_1\pm A_1^*})f_\pm \big\|_{\cH_2} =0.
\lb{2.14w}
\end{equation}
Since $A_2 \pm A_2^*$ are closable and
\begin{equation}
\lim_{n\to\infty} \big\|(A_2\pm A_2^*) V f_{\pm,n} - V(\ol{A_1\pm A_1^*})f_\pm \big\|_{\cH_2} =
\lim_{n\to\infty} \big\|V (A_1\pm A_1^*)f_{\pm,n} - V(\ol{A_1\pm A_1^*})f_\pm \big\|_{\cH_2} =0,   \lb{2.14x}
\end{equation}
one obtains
\begin{equation}
V f_{\pm,n} \in \dom\big(\ol{A_2\pm A_2^*}\big) \, \text{ and } \,
\lim_{n\to\infty} \big\|(A_2\pm A_2^*) V f_{\pm,n} - (\ol{A_2\pm A_2^*}) V f_\pm \big\|_{\cH_2} = 0,
\lb{2.14y}
\end{equation}
and thus,
\begin{equation}
V(\ol{A_1\pm A_1^*})f_\pm = (\ol{A_2\pm A_2^*}) V f_\pm.    \lb{2.14z}
\end{equation}
Upon recalling that $f_\pm \in \dom\big(\ol{A_1\pm A_1^*}\big)$ were arbitrary, this finally implies that
\begin{equation}
V \ol{(A_1 \pm A_1^*)} = \ol{(A_2 \pm A_2^*)} V \, \text{ on } \, \dom\big(\ol{A_1 \pm A_1^*}\big).
\lb{2.14A}
\end{equation}

Next, we recall that $\ol{(A_j + A_j^*)/2}$ and $\ol{(A_j - A_j^*)/(2 i)}$, $j=1,2$, are self-adjoint, and we denote by $\{ E_{\ol{(A_j + A_j^*)/2}}(\la)\}_{\la\in\bbR}$ and
$\{ E_{\ol{(A_j - A_j^*)/(2i)}}(\la)\}_{\la\in\bbR}$,
$j=1,2$, the corresponding family of spectral projections.

Analogously to \eqref{2.14eA}, the families of spectral projections
$\{ E_{A_j}(\nu)\}_{\nu\in\bbC}$ for the normal operators $A_j$, $j=1,2$, are given by
\begin{align}
\begin{split}
& E_{A_j}(\nu) = E_{\ol{(A_j + A_j^*)/2}}(\la) \, E_{\ol{(A_j - A_j^*)/(2i)}}(\mu)
= E_{\ol{(A_j - A_j^*)/(2i)}}(\mu) \, E_{\ol{(A_j + A_j^*)/2}}(\la),    \\
& \hspace*{6.1cm} \nu = \lambda + i \mu \in \bbC, \; \la, \mu \in\bbR, \;\; j=1,2.    \lb{2.14B}
\end{split}
\end{align}
As in the proof of \eqref{2.14}, equations \eqref{2.14A} then yield
\begin{equation}
V E_{\ol{(A_1 + A_1^*)/2}}(\la) = E_{\ol{(A_2 + A_2^*)/2}}(\la) V,  \quad
V E_{\ol{(A_1 - A_2^*)/(2i)}}(\mu) = E_{\ol{(A_1 - A_2^*)/(2i)}}(\mu) V,
\quad \la, \mu \in\bbR.    \lb{2.14C}
\end{equation}
From \eqref{2.14B} and \eqref{2.14C} one then deduces that
\begin{equation}
V E_{A_j}(\nu) = E_{A_j}(\nu) V, \quad \nu \in\bbC, \; j=1,2.   \lb{2.14D}
\end{equation}
With this in hand, the proof is then completed by following the last part
of the proof of Theorem \ref{t2.1} step by step (replacing $\int_{-R}^R $ by $\int_{|\nu|\le R}$, etc.).
\end{proof}

\begin{remark} \lb{r2.4}
We note that the strategy just employed to prove that \eqref{2.14f} implies \eqref{2.14t} is essentially outlined in the special context of similarity and unitarity of normal operators (where $A_2=A_1$) in
\cite[p.\ 219]{We80}. After completing this proof, we became aware of the detailed history of this type of results: Aparently, Fuglede \cite{Fu50} first proved that $VA \subseteq AV$, with $V$ bounded and $A$ normal, implies $VA^* \subseteq A^*V$. This was extended by Putnam \cite{Pu51} to the result at hand, viz., $VA_1 \subseteq A_2V$, with $V$ bounded and $A_j$ normal, $j=1,2$, implies $VA_1^* \subseteq A_2^*V$. Finally, the proof of \eqref{2.14t} we presented is basically due to Rosenblum \cite{Ro58}. For the convenience of the reader (and for some measure of completeness) we decided to keep the short proof of \eqref{2.14t}. For a detailed history of this circle of ideas we refer to \cite[p.\ 9--11]{Pu67}.
\end{remark}

\begin{remark} \lb{r2.5}
For $A_j\in\cB(\cH_j)$ ($A_j$ not necessarily normal), $j=1,2$, and  functions $\phi, \,  \psi$ analytic in
an open neighborhood of the spectra of $A_1$ and $A_2$, one can also use the Dunford--Taylor functional calculus (see, e.g.,
\cite[Sect.\ VII.3]{DS88}) to prove \eqref{2.14k}.
\end{remark}

To make the connection with the polar decomposition of densely defined closed operators in Hilbert spaces, and some of its generalizations, which originally motivated the writing of this paper, we next recall a few facts: Given a densely defined, closed linear operator $S\colon \dom(S) \to \cH_2$,
$\dom(S)\subseteq\cH_1$, the self-adjoint operator $|S|$ is defined as usual by
\begin{equation}
|S|=(S^*S)^{1/2} \geq 0.   \lb{2.19}
\end{equation}
Moreover, we denote by $P_{\cM}$ the orthogonal projection onto the
closed linear subspace $\cM$ of a Hilbert space. The basic facts about the polar decomposition of closed linear operators then read as follows:

\begin{theorem} $($\cite[Sect.\,VI.2.7]{Ka80} $($see also \cite[Sect.\,IV.3]{EE89}$))$  \label{t2.6}
${}$\\
Let $T \colon \dom(T)\subseteq{\mathcal{H}}_1\to {\mathcal{H}}_2$ be a densely defined, closed linear operator. Then,
\begin{align}
T&=U|T|=|T^\ast|U=UT^*U  \, \text{ on } \, \dom(T) = \dom(|T|),  \label{2.20} \\
T^*&=U^*|T^*|=|T|U^*=U^*TU^*   \, \text{ on } \, \dom(T^*) = \dom(|T^*|),   \label{2.21} \\
|T|&=U^*T=T^*U=U^*|T^*|U \, \text{ on } \, \dom(|T|),    \label{2.22} \\
 |T^*|&=UT^*=TU^*=U|T|U^*  \, \text{ on } \, \dom(|T^*|),   \label{2.23}
\end{align}
where
\begin{equation}
U^*U=P_{\ol{{\ran}(|T|)}}=P_{\ol{{\ran}(T^*)}} \, ,  \quad
UU^*=P_{\ol{{\ran}(|T^*|)}}=P_{\ol{{\ran}(T)}} \, .      \label{2.24}
\end{equation}
In particular, $U$ is a partial isometry with initial set $\ol{{\ran}(|T|)}$
and final set $\ol{{\ran}(T)}$.
\end{theorem}

Identifying $V=U$, $A_1=|T|$, $A_2=|T^*|$, Theorem \ref{t2.1} immediately implies the following generalized polar decomposition of $T$ in \eqref{2.20}:

\begin{theorem}   \lb{t2.7}
Let $T\colon \dom(T)\to\cH_2$, $\dom(T)\subseteq\cH_1$ be a densely defined closed operator with polar decomposition as in \eqref{2.20}. In addition, assume that $\phi$ and $\psi$ are Borel functions on $\bbR$ such that $\phi(\lambda) \psi(\lambda) = \lambda$, $\lambda \in \bbR$ and
$\dom(|T|) \subseteq \dom(\psi(|T|))$. Then $T$ has the representation
\begin{equation}
T = \phi(|T^*|) U \psi(|T|)  \, \text{ on } \, \dom(T) = \dom(|T|).   \lb{2.25}
\end{equation}
In particular, for each $\alpha \in [0,1]$,
\begin{equation}
T = |T^*|^\alpha U |T|^{1-\alpha}  \, \text{ on } \, \dom(T) = \dom(|T|).    \lb{2.26}
\end{equation}
\end{theorem}

\begin{remark}  \lb{r2.8}
We note that in the case of a bounded operator $T$, \eqref{2.26} also follows from \eqref{2.20} and a
Stone--Weierstrass-type approximation argument. More precisely, approximating
the functions $\lambda\mapsto \lambda^{\alpha}$
uniformly by a sequence of polynomials on a compact interval
then yields \eqref{2.26} in analogy to the treatment in \cite[p.\ 6--7]{SF70} in connection
with contractions and their associated defect operators.
\end{remark}

\begin{remark}  \lb{r2.9}
The symmetric case $\alpha=1/2$ in \eqref{2.26} (which will play a special role in the following sections) permits a fairly simple and direct proof that we briefly sketch next: Define $R = U^* |T^*|^{1/2}U$.
Then $R\geq 0$ and $R$ is densely defined since $\dom(R) = \dom(|T^*|^{1/2}U) \supseteq \dom(|T^*|U) = \dom(T)$. Thus one concludes that $R$ is symmetric, $R^*\supseteq U^*|T^*|^{1/2}U = R$. In addition,
\begin{equation}
R^* R \supseteq U^*|T^*|^{1/2}UU^*|T^*|^{1/2}U = U^*|T^*|U = |T|,
\end{equation}
using the second relation in \eqref{2.24}, $UU^*=P_{\ol{{\ran}(|T^*|)}}=P_{\ol{{\ran}(|T^*|^{1/2})}}$. Thus
$(\ol{R})^* \ol{R} \supseteq |T|$, and since $|T|$ is self-adjoint and hence maximal, one obtains
$(\ol{R})^* \ol{R}= |T|$. In exactly the same manner one infers
\begin{equation}
R R^* \supseteq U^*|T^*|^{1/2}UU^*|T^*|^{1/2}U = U^*|T^*|U = |T|,
\end{equation}
hence, $\ol{R} (\ol{R})^* \supseteq |T|$ and thus, $(\ol{R})^* \ol{R} = \ol{R} (\ol{R})^* = |T|$.
That is, $\ol R$ is normal and symmetric, and hence self-adjoint. Since in addition, $\ol R \geq 0$,
$\ol R$ is the unique self-adjoint, nonnegative square root of $|T|$,
\begin{equation}
\ol R = (\ol R)^* = |T|^{1/2}.
\end{equation}
Introducing $S=U |T|^{1/2} U^*$, one obtains analogously,
\begin{equation}
\ol S = (\ol S)^* = |T^*|^{1/2}.
\end{equation}
Thus,
\begin{equation}
|T|^{1/2} \supseteq R= U^* |T^*|^{1/2}U, \quad |T^*|^{1/2} \supseteq S = U|T|^{1/2} U^*.
\end{equation}
Next, using also the first relation in \eqref{2.24}, $U^*U=P_{\ol{{\ran}(|T|)}}=P_{\ol{{\ran}(|T|^{1/2})}}$, one infers
\begin{equation}
 U^* |T^*|^{1/2} \supseteq U^*U |T|^{1/2} U^* = |T|^{1/2} U^*,   \quad
 |T|^{1/2} U^* \supseteq U^* |T^*|^{1/2} UU^* = U^* |T^*|^{1/2},
\end{equation}
and hence
\begin{equation}
U^* |T^*|^{1/2} = |T|^{1/2}U^*, \text{ implying } \, U U^* |T^*|^{1/2} = |T^*|^{1/2} = U |T|^{1/2} U^*.
\end{equation}
But then,
\begin{equation}
|T^*|^{1/2} U |T|^{1/2} = U |T|^{1/2} U^* U |T|^{1/2} = U |T| =T,
\end{equation}
as was to be proven.
\end{remark}

\section{Some Applications to Relatively (Form) Bounded and Relatively (Form) Compact
Perturbations of Self-Adjoint Operators} \lb{s3}

The symmetric version
\begin{equation}
T = |T^*|^{1/2} U |T|^{1/2}  \, \text{ on } \, \dom(T) = \dom(|T|)     \lb{3.1}
\end{equation}
of equation \eqref{2.26} permits some applications to relatively (form) bounded and compact perturbations of a self-adjoint operator which we briefly discuss in this section.

The first application concerns circumstances in which relatively bounded perturbations are also relatively form bounded perturbations of a self-adjoint operator. While, as noted in
\cite[Sect.\ VI.1.7]{Ka80}, there seems to be no general connection between relative boundedness and relative form boundedness, such a connection does exist for symmetric perturbations of a self-adjoint operator (cf.\ \cite[Sect.\ VI.1.7]{Ka80} and \cite[Sect.\ X.2]{RS75}). Here we add another result of this type.

To set the stage, we briefly recall the notion of relatively bounded and relatively form bounded perturbations of an operator $A$ in some complex separable Hilbert space $\cH$. For simplicity we will actually assume that $A$ is a closed operator with nonempty resolvent set for the remainder of this section. We recall the following definition:

\begin{definition} \lb{d3.1}
$(i)$ Suppose that $A$ is a closed operator in $\cH$ and $\rho(A)\neq \emptyset$. An operator $B$ in $\cH$ is called {\it relatively bounded} (resp., {\it relatively compact}\,) {\it with respect to $A$} (in short, $B$ is called {\it $A$-bounded} (resp., {\it $A$-compact}\,)), if
\begin{equation}
 \dom(B) \supseteq \dom(A) \, \text{ and } \,  B(A - z I_{\cH})^{-1} \in \cB(\cH) \;
 (\text{resp.,} \in \cB_\infty(\cH)), \quad  z \in \rho(A).   \lb{3.3}
\end{equation}
$(ii)$ Assume, in addition, that $A$ is self-adjoint in $\cH$ and bounded from below, that is,
$A \ge c I_{\cH}$ for some $c\in\bbR$. Then a densely defined and closed operator $B$ in $\cH$ is called {\it relatively form bounded} (resp., {\it relatively form compact}\,) {\it with respect to $A$} (in short, $B$ is called {\it $A$-form bounded} (resp., {\it $A$-form compact}\,)), if
\begin{equation}
 \dom\big(|B|^{1/2}\big) \supseteq \dom\big(|A|^{1/2}\big) \, \text{ and } \,
 |B|^{1/2} ((A + (1 - c) I_{\cH}))^{-1/2} \in \cB(\cH) \; (\text{resp.,} \in \cB_\infty(\cH)).
 \lb{3.4}
\end{equation}
\end{definition}

\medskip

In particular, $B$ is $A$-form bounded (resp., $A$-form compact), if and only if $|B|$ is.

We note that in Definition \ref{d3.1}\,$(ii)$, since $A^{1/2}$ and $|B|^{1/2}$ are closed,
$ \dom\big(|B|^{1/2}\big) \supseteq \dom\big(A^{1/2}\big)$ already implies
$|B|^{1/2} ((A + (1 - c) I_{\cH}))^{-1/2} \in \cB(\cH)$ (cf.\ \cite[Remark IV.1.5]{Ka80}), and hence the first condition in \eqref{3.4} suffices in the relatively form bounded context. In this context we note that in the special case where $B$ is self-adjoint, condition $(i)$ in the definition used by Reed and Simon \cite[p.\ 168]{RS75} already implies their condition $(ii)$. In fact, it implies a bit more, namely, the existence of $\alpha\geq 0$ and 
$\beta\geq 0$, such that 
\begin{equation}
\big|\big(|B|^{1/2}f, \sgn(B) |B|^{1/2}f\big)_{\cH}\big| \leq \big\||B|^{1/2}f\big\|_{\cH}^2 
\leq \alpha \big\||A|^{1/2}f\big\|_{\cH}^2 + \beta \|f\|_{\cH}^2, \quad 
f\in\dom\big(|A|^{1/2}\big).  
\end{equation}  

Similarly, if $B$ is closed (in fact, closability of $B$ suffices) in Definition \ref{d3.1}\,$(i)$, then the first condition $\dom(B) \supseteq \dom(A)$ in \eqref{3.3} already implies 
$B(A - z I_{\cH})^{-1} \in \cB(\cH)$, $z \in \rho(A)$, and hence the $A$-boundedness of $B$.

Using the polar decomposition of $B$ (i.e., $B=U|B|$), one observes that $B$ is $A$-bounded (resp., $A$-compact) if and only if $|B|$ is $A$-bounded (resp., $A$-compact).

We recall that in connection with relative boundedness, \eqref{3.3} can be replaced by the condition
\begin{align}
\begin{split}
& \dom(B) \supseteq \dom(A), \, \text{ and there exist numbers
$a\ge 0$, $b\ge 0$ such that} \\
& \|Bf\|_{\cH} \le a \|Af\|_{\cH} + b \|f\|_{\cH} \, \text{ for all $f\in\dom(A)$,}    \lb{3.5}
\end{split}
\end{align}
or equivalently, by
\begin{align}
\begin{split}
& \dom(B) \supseteq \dom(A), \, \text{ and there exist numbers $\wti a\ge 0$, $\wti b\ge 0$ such that} \\
& \|Bf\|^2_{\cH} \le {\wti a}^2 \|Af\|^2_{\cH} + {\wti b}^2 \|f\|^2_{\cH} \, \text{ for all $f\in\dom(A)$.}     \lb{3.6}
\end{split}
\end{align}

Clearly, \eqref{3.6} implies \eqref{3.5}  with $a=\wti a$, $b=\wti b$ and conversely, \eqref{3.5} implies \eqref{3.6} with ${\wti a}^2= (1 + \varepsilon) a^2$, ${\wti b}^2 = (1 + \varepsilon^{-1})b^2$ for each
$\varepsilon >0$. We also note that if $A$ is self-adjoint and bounded from below, the number $\alpha$ defined by
\begin{equation}
\alpha = \lim_{\mu\uparrow \infty} \big\|B(A+\mu I_{\cH})^{-1}\big\|_{\cB(\cH)}
= \lim_{\mu\uparrow \infty} \big\||B|(A+\mu I_{\cH})^{-1}\big\|_{\cB(\cH)}     \lb{3.7}
\end{equation}
equals the greatest lower bound (i.e., the infimum) of the possible values for $a$ in \eqref{3.5} (resp., for 
$\wti a$ in \eqref{3.6}). This number $\alpha$ is called the $A$-bound of $B$. Similarly, we call
\begin{equation}
\beta = 
\lim_{\mu\uparrow \infty} \big\||B|^{1/2}\big(|A|^{1/2}+\mu I_{\cH}\big)^{-1}\big\|_{\cB(\cH)}     \lb{3.7a}
\end{equation}
the $A$-form bound of $B$ (resp., $|B|$). If $\alpha =0$ in \eqref{3.7} (resp., $\beta =0$ in \eqref{3.7a}) then $B$ is called {\it infinitesimally bounded} (resp., {\it infinitesimally form bounded}\,) with respect to $A$.

We then have the following result:

\begin{theorem} \lb{t3.2}
Assume that $A\ge 0$ is self-adjoint in $\cH$. \\
$(i)$ Let $B$ be a closed, densely defined operator in $\cH$ and suppose that
$\dom(B)\supseteq\dom(A)$. Then $B$ is $A$-bounded and hence \eqref{3.5} holds for some constants $a\geq 0$, $b\geq 0$. In addition, $B$ is also $A$-form bounded,
\begin{equation}
 |B|^{1/2} (A +I_{\cH})^{-1/2} \in \cB(\cH).
\lb{3.8}
\end{equation}
More specifically,  
\begin{equation}
\big\||B|^{1/2}(A+I_{\cH})^{-1/2}\big\|_{\cB(\cH)} \le (a+b)^{1/2},    \lb{3.8a}
\end{equation}
and hence, if $B$ is $A$-bounded with $A$-bound $\alpha$ strictly less than one, $0\leq \alpha<1$ $($cf.\ \eqref{3.7}$)$, then $B$ is also $A$-form bounded with $A$-form bound $\beta$ strictly less than one, $0\leq \beta <1$ $($cf.\ \eqref{3.7a}$)$. In particular, if $B$ is infinitesimally bounded with respect to $A$, then $B$ is infinitesimally form bounded with respect to $A$. \\
$(ii)$ Suppose that $B$ is closed and densely defined in $\cH$, that 
$\dom(B) \cap \dom(B^*)\supseteq\dom(A)$, and hence \eqref{3.5} holds for some constants $a\geq 0$, $b\geq 0$. Then also $B^*$ is $A$-bounded, and hence \eqref{3.5} with $B$ replaced by $B^*$ holds for some constants $a^*\geq 0$, $b^*\geq 0$. In particular, 
\begin{equation}
|B^*|^{1/2} (A +I_{\cH})^{-1/2} \in \cB(\cH) \, \text{ and } \, 
\big\||B^*|^{1/2}(A+I_{\cH})^{-1/2}\big\|_{\cB(\cH)} \le (a^* + b^*)^{1/2}.    \lb{3.8b}
\end{equation}
Moreover, one has
\begin{align}
& \ol{(A +  I_{\cH})^{-1/2} B (A +  I_{\cH})^{-1/2}}, \,\,
\ol{(A + I_{\cH})^{-1/2} B^* (A + I_{\cH})^{-1/2}} \in \cB(\cH),   \lb{3.9}  \\
& \big\|\ol{(A+I_{\cH})^{-1/2}B(A+I_{\cH})^{-1/2}}\big\|_{\cB(\cH)}
 \le (a^* + b^*)^{1/2}(a+b)^{1/2},   \lb{3.9a}  \\
& \big\|\ol{(A+I_{\cH})^{-1/2}B^*(A+I_{\cH})^{-1/2}}\big\|_{\cB(\cH)}
 \le (a^* + b^*)^{1/2}(a+b)^{1/2}.    \lb{3.9b} 
\end{align}
\end{theorem}
\begin{proof}
$(i)$ Equation \eqref{3.5} implies
\begin{equation}
\big\|B(A+I_{\cH})^{-1}\big\|_{\cB(\cH)} = \big\||B|(A+I_{\cH})^{-1}\big\|_{\cB(\cH)} \le a+b,   \lb{3.10}
\end{equation}
using the polar decomposition $B=U|B|$ of $B$ (cf.\ \eqref{2.20}). Since also
\begin{equation}
\big\|(\ol{A+I_{\cH})^{-1}|B|}\big\|_{\cB(\cH)} = \big\|\big[|B|(A+I_{\cH})^{-1}\big]^*\big\|_{\cB(\cH)} \le a+b,    \lb{3.11}
\end{equation}
the proof of Theorem\ X.18 in \cite{RS75} yields that
\begin{equation}
\big\|\ol{(A+I_{\cH})^{-1/2}|B|(A+I_{\cH})^{-1/2}}\big\|_{\cB(\cH)}
= \big\|\big[|B|^{1/2}(A+I_{\cH})^{-1/2}\big]^*|B|^{1/2}(A+I_{\cH})^{-1/2}\big\|_{\cB(\cH)} \le a+b,    \lb{3.12}
\end{equation}
using complex interpolation. Indeed, one considers $C^\infty(A) =\bigcap_{n\in\bbN} \dom(A^n)$, introduces $\cH_p(A)$, $p\in\bbR$, as the completion of $C^\infty(A)$ with respect to the norm
$\|f\|_p=\|(A+I_{\cH})^{p/2} f\|_{\cH}$, $f\in C^\infty(A)$. Then $\cH^*_p=\cH_{-p}$, $p\in\bbR$. Given $p_0, p_1 \in\bbR$ and $p_t=t p_0 + (1-t) p_1$, $t\in [0,1]$, one can prove that $\cH_{p_t}$ are interpolating spaces between $\cH_{p_0}$ and $\cH_{p_1}$. Given   $m,n \in \bbN$, an operator $C\colon C^\infty(A) \to \cH$ extends to a bounded operator from $\cH_m$ to $\cH_{-n}$, if and only if
 \begin{equation}
 (A+I_{\cH})^{-n/2} C (A+I_{\cH})^{-m/2} \in\cB(\cH).   \lb{3.12a}
 \end{equation}
 The case at hand in \eqref{3.12} then alludes to the special situation $m=n=1$ in 
 \eqref{3.12a}.

Since
\begin{equation}
\|T^*T\|_{\cB(\cH)}=\|TT^*\|_{\cB(\cH)}=\|T\|^2_{\cB(\cH)}=\|T^*\|^2_{\cB(\cH)} \,
\text{  for all $T\in\cB(\cH)$,}
\end{equation}
\eqref{3.12} yields the estimate \eqref{3.8a}.

To prove the remaining assertions in item $(i)$ one substitutes $(A-\mu)^{-1}f$ in place of $f$ in \eqref{3.5} and obtains 
\begin{equation}\label{3.19a}
\|B(A+\mu)^{-1}f\|_{\cH} \le a\|A(A+\mu)^{-1}f\|_{\cH} + b \|(A+\mu)^{-1}f\|_{\cH}
\le(a+b/\mu)\|f\|_{\cH}, \quad f \in \dom(A), \; \mu>0,  
\end{equation}
and hence,  
\begin{equation}
\|Bf\|_{\cH} \leq [a+(b/\mu)] \|Af\|_{\cH} + (a\mu + b)\|f\|_{\cH}, \quad f \in \dom(A), \;
 \mu>0. 
\end{equation}
Similarly, the inequality
\begin{equation}
\big\||B|^{1/2}(A+\mu I_{\cH})^{-1/2}\big\|_{\cB(\cH)} \leq [a+(b/\mu)]^{1/2}
\end{equation}
which follows from \eqref{3.19a} in the same manner as
\eqref{3.8a} follows from \eqref{3.4} (i.e., by the same interpolation argument), implies
\begin{equation}
\big\||B|^{1/2}f\|_{\cH} \leq [a+(b/\mu)]^{1/2} \big\||A|^{1/2}f\big\|_{\cH} 
+ (a\mu + b)^{1/2}\|f\|_{\cH}, \quad f \in \dom\big(|A|^{1/2}\big), \; \mu>0.
\end{equation} 
$(ii)$ By symmetry of our hypotheses one obtains \eqref{3.8b}. 

Next, using the generalized polar decomposition $B = |B^*|^{1/2} U |B|^{1/2}$ (cf.\ 
\eqref{3.1}), one thus obtains from \eqref{3.8a} and \eqref{3.8b} that
\begin{align}
\big\|\ol{(A+I_{\cH})^{-1/2}B(A+I_{\cH})^{-1/2}}\big\|_{\cB(\cH)}
&= \big\|\ol{(A+I_{\cH})^{-1/2}|B^*|^{1/2} U |B|^{1/2}(A+I_{\cH})^{-1/2}}\big\|_{\cB(\cH)}  \no \\
&= \big\|\big[|B^*|^{1/2}(A+I_{\cH})^{-1/2}\big]^* U |B|^{1/2}(A+I_{\cH})^{-1/2}\big\|_{\cB(\cH)} \no \\
& \le (a^* + b^*)^{1/2} \|U\|_{\cB(\cH)} (a+b)^{1/2}  \no \\
& \le (a^* + b^*)^{1/2}(a+b)^{1/2}.   \lb{3.15}
\end{align}
Since
\begin{equation}
\Big[\,\ol{(A+I_{\cH})^{-1/2}B(A+I_{\cH})^{-1/2}}\,\Big]^*
= \ol{(A+I_{\cH})^{-1/2}B^*(A+I_{\cH})^{-1/2}},    \lb{3.16}
\end{equation}
this completes the proof of \eqref{3.9}--\eqref{3.9b}.
\end{proof}

Theorem \ref{t3.2} extends Theorem\ 1.38 in \cite[Sect.\ VI.1.7]{Ka80} and Theorem\ X.18 in \cite{RS75} since $B$ is not assumed to be self-adjoint or symmetric.

We note in connection with the hypotheses in Theorem \ref{t3.2}, that if $B$ is $A$-bounded, then $B^*$ need not be $A^*$-bounded nor $A$-bounded (we recall that $A^*=A\ge 0$ in our present case). Indeed, the following simple example illustrates this point:

\begin{example} \lb{e3.3}
Consider the densely defined closed operators in $L^2((0,1); dx)$:
\begin{align}
\begin{split}
& T_{min} = -\f{d^2}{dx^2}, \quad
\dom(T_{min}) = \{g \in L^2((0,1); dx)\,|\, g, g' \in AC([0,1]); \\
& \hspace*{3cm} g(0)=g'(0)=g(1)=g'(1)=0; \, g'' \in L^2((0,1); dx)\} \\
& T_{F} = -\f{d^2}{dx^2}, \quad
\dom(T_{F}) = \{g \in L^2((0,1); dx)\,|\, g, g' \in AC([0,1]); \\
& \hspace*{3cm} g(0)=g(1)=0; \, g'' \in L^2((0,1); dx)\}    \lb{3.17} \\
& T_{max} = -\f{d^2}{dx^2}, \quad
\dom(T_{max}) = \{g \in L^2((0,1); dx)\,|\, g, g' \in AC([0,1]); \, g'' \in L^2((0,1); dx)\},
\end{split}
\end{align}
where $T_{F}^*=T_F > 0$ is the Friedrichs extension of the minimal operator $T_{min}$. $($Here
$AC(\cI)$ denotes the set of absolutely continuous functions on the interval $\cI\subset \bbR$.$)$ Then the maximal operator $T_{max}$ is $T_F$-bounded since $\dom(T_{max}) \supset \dom(T_F)$ and both operators are closed $($cf.\  \cite[Remark IV.1.5]{Ka80}$)$, but $T_{max}^*=T_{min}$ is not $T_F$-bounded since
$\dom(T_{min})$ is strictly contained in $\dom(T_F)$.
\end{example}

Before we turn to relatively (form) compact perturbations, we recall a useful interpolation result:

\begin{theorem}  $($\cite[Theorem\ IV.1.11]{KPS82}$)$  \lb{t3.4}
Suppose $A \ge I_{\cH}$ and $B \ge I_{\cH}$ are self-adjoint operators in $\cH$ with
$\dom(B) \supseteq \dom(A)$. If
\begin{equation}
\|Bf\|_{\cH} \le \|Af\|_{\cH} \, \text{ for all $f \in \dom(A)$},   \lb{3.18}
\end{equation}
then for all $\alpha\in [0,1]$, one has
\begin{equation}
\big\|B^\alpha f\big\|_{\cH} \le \big\|A^\alpha f\big\|_{\cH} \, \text{ for all $f\in\dom\big(A^\alpha\big)$.}
\lb{3.19}
\end{equation}
\end{theorem}

\begin{theorem} \lb{t3.5}
Assume that $A\ge 0$ is self-adjoint in $\cH$. \\
$(i)$ Let $B$ be a densely defined, closed operator in $\cH$ and suppose that
$\dom(B)\supseteq\dom(A)$. In addition, assume that $B$ is $A$-compact. Then $B$ is also $A$-form compact,
\begin{equation}
|B|^{1/2} (A+I_{\cH})^{-1/2} \in\cB_\infty (\cH).     \lb{3.22}
\end{equation}
$(ii)$ Suppose that $B$ is densely defined and closed in $\cH$ and that
$\dom(B)\cap\dom(B^*)\supseteq \dom(A)$. In addition, assume that $B$ or $B^*$ is 
$A$-compact. Then 
\begin{equation}
 \ol{(A+I_{\cH})^{-1/2} B (A+I_{\cH})^{-1/2}}, \,\,
  \ol{(A+I_{\cH})^{-1/2} B^* (A+I_{\cH})^{-1/2}}
  \in \cB_\infty (\cH).     \lb{3.23}
\end{equation}
\end{theorem}
\begin{proof}
$(i)$ An elementary computation shows that \eqref{3.6} implies
\begin{equation}
\|Bf\|_{\cH}^2 = \|B^*f\|_{\cH}^2 = \big\|\big({\wti a}^2 A^2 + {\wti b}^2 I_{\cH}\big)^{1/2}f\big\|^2_{\cH},
\quad f\in\dom(A).    \lb{3.24}
\end{equation}
Replacing $|B|$ by $|B|+I_{\cH}$ and $A$ by $A+I_{\cH}$, Theorem \ref{t3.4} implies
\begin{equation}
\dom\big(|B|^\alpha\big) \supseteq \dom\big(A^\alpha\big), \quad \alpha \in [0,1].    
\lb{3.25}
\end{equation}
As a result, terms of the type
\begin{align}
\ol{(A+I_{\cH})^{-z} |B| (A+I_{\cH})^{-1+z}}
&= \ol{\big[(A+I_{\cH})^{-z}|B|^{z} \big]} \big[|B|^{1-z} (A+I_{\cH})^{-1+z} \big]  \no \\
& =  \big[|B|^{\ol z} (A+I_{\cH})^{-\ol z}\big]^* \big[|B|^{1-z} (A+I_{\cH})^{-1+z} \big],   \lb{3.26} \\
& \hspace*{3.55cm}  z\in\bbC, \; \Re(z) \in [0,1],   \no
\end{align}
are well-defined as operators in $\cB(\cH)$.

Next we allude to the complex interpolation proof of the Lemma on p.\ 115 in \cite{RS78}. The proof of this lemma and equation \eqref{3.26} yield that
\begin{equation}
\ol{(A+I_{\cH})^{-z} |B| (A+I_{\cH})^{-1+z}} \in \cB_\infty (\cH) \, \text{ for all $z\in\bbC$, $\Re (z) \in (0,1)$.}
\lb{3.27}
\end{equation}
Taking $z=1/2$ in \eqref{3.27} one concludes
\begin{equation}
\ol{(A+I_{\cH})^{-1/2} |B| (A+I_{\cH})^{-1/2}} \in \cB_\infty (\cH),
\lb{3.28}
\end{equation}
and the latter is then equivalent to
\begin{equation}
 |B|^{1/2} (A+I_{\cH})^{-1/2} \in \cB_\infty (\cH)   \lb{3.29}
\end{equation}
(since $T^*T\in\cB_\infty(\cH)$ is equivalent to $T\in\cB_\infty(\cH)$). 

$(ii)$ Using again the generalized polar decomposition \eqref{3.1} of $B$, 
$B=|B^*|^{1/2} U |B|^{1/2}$, one obtains
\begin{align}
& \ol{(A+I_{\cH})^{-1/2} B (A+I_{\cH})^{-1/2}}
= \ol{\big[(A+I_{\cH})^{-1/2}|B^*|^{1/2}\big]} U \big[ |B|^{1/2} (A+I_{\cH})^{-1/2}\big]
\no \\
& \quad =  \big[ |B^*|^{1/2} (A+I_{\cH})^{-1/2}\big]^*  U
\big[ |B|^{1/2} (A+I_{\cH})^{-1/2}\big] \in \cB_\infty (\cH),    \lb{3.30}
\end{align}
since both square brackets in the last equality in \eqref{3.30} are bounded operators and by hypothesis at least one of them is compact. Employing \eqref{3.16} again completes the proof of \eqref{3.23}.
\end{proof}

Equation \eqref{3.23} extends \cite[Problem\ 73\,(a), p.\ 373]{RS78}, since $B$ is not assumed to be symmetric.

In a completely analogous manner one proves membership of  the operators in \eqref{3.23} in the Schatten--von Neumann classes $\cB_p(\cH)$, $p\ge 1$; we omit further details.

\begin{remark} \lb{r3.6}
We conclude this section by recalling a well-known example, where $A$ and $B$ are self-adjoint,
$B$ is $A$-form bounded and even $A$-form compact, but $B$ is not $A$-bounded (let alone $A$-compact): Denote by $\cR$ the class of Rollnik potentials in $\bbR^3$, that is,
\begin{equation}
\cR = \bigg\{V\colon\bbR^3 \to \bbC \,\bigg|
\int_{\bbR^6} d^3x \, d^3y \, \f{|V(x)| |V(y)|}{|x-y|^2} < \infty\bigg\},
\end{equation}
and by $H_0$ the $L^2(\bbR;d^3x)$-realization of (minus) the Laplacian $-\Delta$ defined on the Sobolev space $H^{2,2}(\bbR^3)$. Then there exist potentials $0 \leq V_0 \in L^1(\bbR^3;d^3x)\cap\cR$
such that
\begin{equation}
V_0^{1/2} \big(H_0 + I_{L^2(\bbR^3;d^3x)}\big)^{-1/2} 
\in \cB_4\big(L^2(\bbR^3;d^3x)\big),
\end{equation}
(cf.\ Simon \cite{Si71}, Theorem I.22 and Example 4 in Sect.\ I.6) and hence $V_0$ is $H_0$-form compact, but
\begin{equation}
\dom(V_0) \cap \dom(H_0) =\{0\},
\end{equation}
and thus $V_0$ is not $H_0$-bounded.
\end{remark}

\section{Some Applications to Maximally Sectorial Operators} \lb{s4}

In this section we relax the condition that $A$ is self-adjoint and study maximally sectorial operators
$A$ instead.

We recall that $A$ is called {\it accretive} if the numerical range of $A$ (i.e., the set 
$\{(f,Af)_{\cH}\in\bbC \,|\, f \in \dom(A), \, \|f\|_{\cH}=1\}$) is a subset of the closed right complex half-plane. $A$ is called {\it $m$-accretive} if $A$ is a closed and maximal accretive operator (i.e., $A$ has no proper accretive extension). Moreover,  $A$ is called an {\it $m$-sectorial} operator with a {\it vertex} $0$ and a corresponding {\it semi-angle} $\theta\in [0,\pi/2)$ if $A$ is a maximal accretive, closed (and hence densely defined) operator, and the numerical range of $A$ is contained in a sector $|\arg(z)|\le \theta < (\pi/2)$ in the complex $z$-plane.

We also recall that an equivalent definition of an $m$-accretive operator $A$ in $\cH$ is
\begin{equation}
(A+\zeta I_{\cH})^{-1} \in \cB(\cH), \quad
\|(A+\zeta I_{\cH})^{-1}\| \leq \f{1}{\Re(\zeta)}, \quad \Re(\zeta) > 0.   \lb{4.1}
\end{equation}

With $A$ assumed to be $m$-sectorial, one associates the quadratic form
\begin{equation}
\gt'_A [f,g]= (f,Ag)_{\cH}, \quad f, g \in \dom ({\gt}'_A) =\dom (A).
\end{equation}
The form $\gt'_A$ is closable (cf.\ \cite[Theorem VI.1.27]{Ka80})
and according to the first representation theorem
(see, e.g., \cite[Theorem VI.2.1]{Ka80}), $A$ is associated with its closure
$\gt_A = \overline{\gt'_A}$, that is,
\begin{equation}
\dom(A)\subseteq \dom(\gt_A) \, \text{ and } \,
\gt_A [f,g] = (f,Ag)_{\cH}, \quad   f\in\dom(\gt_A), \, g\in\dom(A).
\end{equation}

Denoting by $\gt^*$ the adjoint form of a sesqulinear form $\gt$ in $\cH$,
\begin{equation}
t^*[f,g] = \ol{\gt[g,f]}, \quad f,g \in \dom(\gt^*) =\dom(\gt),
\end{equation}
the form $\gt_{A_{\gR}} = (\gt_A + \gt_{A}^*)/2$ is
closed and nonnegative on $\dom(\gt_{A_{\gR}}) = \dom(\gt_A)$. We denote by $A_{\gR}\ge 0$ the self-adjoint operator uniquely associated with $\gt_{A_{\gR}}$, that is,
\begin{equation}
\dom(A_{A_{\gR}})\subseteq \dom(\gt_A) \, \text{ and } \,
\gt_{A_{\gR}}[f,g]=(f,A_\gR g)_{\cH}, \quad  f\in\dom(\gt_A), \, g\in\dom(A_{\gR}).
\end{equation}
By the second representation theorem (cf.\ \cite[Theorem VI.2.23]{Ka80})
\begin{equation}
\gt_{A_{\gR}}[f,g] = \big(A^{1/2}_{\gR} f, A^{1/2}_{\gR} g\big)_{\cH}, \quad
f,g\in\dom(\gt_{A_{\gR}}) = \dom \big(A^{1/2}_{\gR}\big).
\end{equation}
We denote by $A^{1/2}$ the unique $m$-sectorial square root of $A$, and recall that
\begin{equation}
(A^*)^{1/2} = \big(A^{1/2}\big)^*.
 \end{equation}
It should be emphasized that in general,
\begin{equation}
\dom \big(A^{1/2}\big) \ne \dom \big((A^*)^{1/2}\big).
 \end{equation}
However, if in fact, $\dom \big(A^{1/2}\big)=\dom \big((A^*)^{1/2}\big)$,
then one can obtain the analog of the second representation theorem for densely defined, closed, sectorial forms. For this purpose we next recall the following results:

\begin{theorem} \lb{t4.1} $($\cite[Theorem VI.3.2]{Ka80}$)$
Let $A$ be $m$-sectorial  in $\cH$ with a vertex $0$ and semi-angle $\theta\in[0,\pi/2)$.
Then $A_{\gR} \ge 0$ and there exists a bounded self-adjoint operator $X\in\cB(\cH)$
such that $\|X\|_{\cB(\cH)}\le \tan(\theta)$ and
\begin{equation}\label{4.8}
A=A^{1/2}_{\gR}(I_{\cH}+i X)A^{1/2}_{\gR},  \quad A^* = A^{1/2}_{\gR} (I_{\cH} - i X)A^{1/2}_{\gR}.
\end{equation}
\end{theorem}

\begin{lemma} \lb{l4.2} $($cf.\ \cite{Ka61}, \cite[Theorem VI.3.2]{Ka80} and \cite[Theorem IV.2.10]{EE89}$)$
Let $A$ be $m$-sectorial in $\cH$ with a vertex $0$ and assume that
 \begin{equation}\label{4.9}
\dom \big(A^{1/2}\big) = \dom \big((A^*)^{1/2}\big).
\end{equation}
Then the sesquilinear form
   \begin{equation}
\gt_A[f,g] = \big((A^*)^{1/2} f, A^{1/2} g\big)_{\cH}, \quad f,g
\in \dom \big(A^{1/2}\big) = \dom \big((A^*)^{1/2}\big),   \lb{4.10}
\end{equation}
is sectorial and closed. In particular,
\begin{equation}
\dom(\gt_A)=\dom \big(A^{1/2}\big)= \dom \big((A^*)^{1/2}\big)=\dom \big(A^{1/2}_{\gR}\big)
=\dom(\gt_{A_{\gR}}).      \lb{4.11}
\end{equation}
\end{lemma}
\begin{proof} Although this result is known (cf.\ Kato \cite{Ka61}), we thought it might be of some
interest to present an alternative proof. Since
 \begin{align}
 \begin{split}
\gt_{A+I_{\cH}}[f,g] &= \big((A^* + I_{\cH})^{1/2} f, (A + I_{\cH})^{1/2} g\big)_{\cH}
= \big((A^*)^{1/2} f, A^{1/2} g\big)_{\cH} + (f,g)_{\cH}, \\
& = \gt_A[f,g] + (f,g)_{\cH}, \quad f,g \in \dom \big(A^{1/2}\big) = \dom\big((A^*)^{1/2}\big),
\end{split}
\end{align}
it suffices to consider $\gt_{A+I_{\cH}}$ instead of $\gt_A$ in the remainder of this proof.
Since  $\dom \big(A^{1/2}\big)=\dom \big((A^*)^{1/2}\big)$, and the operators
$A^{1/2}$ and $(A^*)^{1/2}$ are closed,  one concludes that the operator $Y$ defined below,
satisfies
\begin{equation}\label{4.12}
Y = (A + I_{\cH})^{1/2}(A^* + I_{\cH})^{-1/2}\in \cB(\cH) \,  \text{ and } \,
Y^{-1} = (A^* + I_{\cH})^{1/2}(A + I_{\cH})^{-1/2}\in \cB(\cH).
 \end{equation}

Next we show that the operator $Y$ is accretive. Since $A$ is $m$-accretive, one gets
\begin{align}
\begin{split}
2\Re(Y) &= Y + Y^* = (A + I_{\cH})^{1/2}(A^* + I_{\cH})^{-1/2} + (A + I_{\cH})^{-1/2}(A^* + I_{\cH})^{1/2} \\
&=  (A + I_{\cH})^{1/2} \big[(A^* + I_{\cH})^{-1} + (A + I_{\cH})^{-1}\big] (A^* + I_{\cH})^{1/2} \ge 0.
\lb{4.13}
\end{split}
\end{align}
Similarly one obtains
\begin{align}
\begin{split}
2i \Im(Y) &= (Y - Y^*) = (A + I_{\cH})^{1/2}(A^* + I_{\cH})^{-1/2}
-  (A + I_{\cH})^{-1/2}(A^* + I_{\cH})^{1/2}  \\
&=  2i(A + I_{\cH})^{1/2} \Im\big((A^* + I_{\cH})^{-1}\big) \big((A^* + I_{\cH})^{1/2}\big).  \label{4.14}
\end{split}
\end{align}
Combining \eqref{4.13} with  \eqref{4.14} one concludes that $Y$ is a bounded $\theta$-sectorial operator, because so is $(A^* + I_{\cH})^{-1}$.

Using this fact and $0\notin \sigma(Y)$, we next show  that
$0\notin \sigma(\Re(Y))$. Indeed, assuming the contrary,
$0\in\sigma(\Re(Y))$, we get $0\in\sigma\big((\Re(Y))^{1/2}\big)$. Then
there exists a sequence $\{f_n\}_{n\in\bbN}\subset\cH$, $\|f_n\|=1$, $n\in\bbN$, such that
$(\Re(Y))^{1/2} f_n \underset{n\to\infty}{\longrightarrow} 0$.

On the other hand, by Theorem \ref{t4.1}, $Y$ admits the representation
\begin{equation}
Y  =  (\Re(Y))^{1/2}(I+iK)(\Re(Y))^{1/2},
\end{equation}
where $K=K^*$ and $\|K\|\le \tan(\theta)$. Consequently,
\begin{equation}
\|Y f_n\| = \big\|(\Re(Y))^{1/2}(I + iK) (\Re(Y))^{1/2} f_n\big\| \le
C \big\|(\Re(Y))^{1/2} f_n \big\| \underset{n\to\infty}{\longrightarrow} 0.
\end{equation}
Thus,  $0\in\sigma(Y)$, a contradiction. Hence, the operator $(\Re(Y))^{1/2}(A+I)^{1/2}$ is closed as
$(A+I)^{1/2}$ is and $\Re(Y)$ is boundedly invertible.

Moreover, using \eqref{4.12} one obtains
 \begin{align}
\gt_{A_{\gR}+I_{\cH}}[f,g] &= 2^{-1}\big((A^* + I_{\cH})^{1/2} f, (A + I_{\cH})^{1/2} g\big)_{\cH}
+ 2^{-1}\big((A + I_{\cH})^{1/2} f, (A^* + I_{\cH})^{1/2} g\big)_{\cH} \no \\
& = 2^{-1} \big((Y + Y^*)(A^* + I_{\cH})^{1/2} f, (A^* + I_{\cH})^{1/2} g\big)_{\cH}  \\
& = \big((\Re(Y))^{1/2}(A^* + I_{\cH})^{1/2} f, (\Re(Y))^{1/2}(A^* + I_{\cH})^{1/2}g\big)_{\cH},  \\
 &  \hspace*{3.7cm}   f,g \in \dom \big(A^{1/2}\big) = \dom \big((A^*)^{1/2}\big).  \no
\end{align}
Since the operator $(A^* + I_{\cH})^{1/2}$ is closed and $\Re(Y)$ is boundedly invertible, also the operator $(\Re(Y))^{1/2}(A^* + I_{\cH})^{1/2}$ is closed, and hence the form $\gt_{A_{\gR}+I_{\cH}}$ is closed too (cf.\ \cite[Problem III.5.7, Example VI.1.13]{Ka80}).
 \end{proof}

 \begin{remark}  \lb{r4.3}
Let $A=\diag \, (it_0, -it_0), \ t_0\in \bbR\backslash\{0\}$. Then $A$ is
maximal accretive operator in ${\bbC}^2$ and $0\in\rho(A)$,
although $A_{\gR}=0$ and hence $0\in\sigma(A_{\gR})$. This simple
example shows that the assumption on $A$ to be $m$-sectorial is
important in proving the implication $0\in\rho(A) \Longrightarrow
0\in\rho(A_{\gR})$.
\end{remark}

\begin{corollary} \lb{c4.4}
Let $A$ be $m$-sectorial in $\cH$ with a vertex $0$ and assume that
 \begin{equation}\label{4.18a}
\dom \big(A^{1/2}\big) = \dom \big((A^*)^{1/2}\big)
\end{equation}
and that
\begin{equation}
0 \in \rho(A).    \lb{4.19}
\end{equation}
Then, in addition to \eqref{4.10} and \eqref{4.11},  there exists an $\varepsilon_0>0$ such that the following inequalities hold:
\begin{align}
\gt_{A_{\gR}}[f,f]&= \Re\big(\big((A^*)^{1/2}f, A^{1/2}f\big)_{\cH}\big) \ge \varepsilon_0
\max \big\{\big\|A^{1/2}f\big\|^2_{\cH}, \big\|(A^*)^{1/2}f\big\|^2_{\cH} \big\}, \quad
 f\in\dom(\gt_{A_{\gR}}),     \label{4.20}  \\
\gt_{A_{\gR}}[f, f] &= \Re\big(\big((A^*)^{1/2}f, A^{1/2} f\big)_{\cH}\big)
\ge\varepsilon_0 \big\|A^{1/2}f\big\|_{\cH} \big\|(A^*)^{1/2}f\big\|_{\cH}, \quad f\in\dom(\gt_{A_{\gR}}).
\label{4.21A}
\end{align}
\end{corollary}
 \begin{proof}
As in the proof of Lemma \ref{l4.2}, we may write $(A)^{1/2} = Y(A^*)^{1/2}$, where $Y$ is
$m$-accretive and  $0\in\rho(\Re(Y))$. In this context we note that we replaced $A+I_{\cH}$ and
$A^*+I_{\cH}$ by $A$ and $A^*$ in the definition of $Y$ in \eqref{4.12} to arrive at the operator
\begin{equation}
Y = A^{1/2}(A^*)^{-1/2}\in \cB(\cH),
\end{equation}
which is possible due to the hypothesis $0 \in \rho(A)$ (implying $0 \in \rho(A^*)$). Therefore,
\begin{align}
\begin{split}
\gt_{A_{\gR}}[f, f] &= \Re\big(\big((A^*)^{1/2}f, (A)^{1/2} f\big)_{\cH}\big) =
\big((A^*)^{1/2}f, \Re(Y) (A^*)^{1/2}f\big)_{\cH}  \\
& \ge \varepsilon_1 \big\|(A^*)^{1/2}f\big\|^2_{\cH},  \quad f\in\dom(\gt_{A_{\gR}}),
\end{split}
\end{align}
where $\varepsilon_1 = \inf \big(\sigma(\Re(Y))\big)$.  Similarly, one gets
\begin{equation}
\Re\big((A^*)^{1/2}f, (A)^{1/2} f\big)_{\cH} \ge \varepsilon_2 \big\|(A)^{1/2}f\big\|^2_{\cH},
\quad f\in\dom(\gt_{A_{\gR}}),
\end{equation}
where $\varepsilon_2 =\inf\big(\sigma\big(\Re\big(Y^{-1}\big)\big)\big)$.
Setting $\varepsilon_0 = \min\{\varepsilon_1, \varepsilon_2\}$ one
arrives at \eqref{4.20}. Inequality \eqref{4.21A} is then immediate from \eqref{4.20}.
\end{proof}

\begin{remark}  \lb{r4.5}
$(i)$ Inequality \eqref{4.21A} is mentioned in \cite{Ka62}, and because of inequality \eqref{4.21A},
$A^{1/2}$ and $(A^*)^{1/2}$ are said to have an acute angle.  \\
$(ii)$ In general, if $A$ is $m$-accretive (i.e., without assuming \eqref{4.18a} and
\eqref{4.19}), Kato \cite{Ka61a} proved
\begin{equation}
\dom (A^{\alpha}) = \dom ((A^*)^{\alpha}), \quad \alpha\in (0,1/2),
\end{equation}
and that the (right-hand) inequality in \eqref{4.21A} holds with $1/2$ replaced by $\alpha$ (cf.\ also
\cite{Ka62}, \cite[Theorem IV.5.1]{SF70}), that is, there exists an $\varepsilon_0(\alpha)>0$ such that
\begin{equation}
\Re\big(\big((A^*)^{\alpha}f, A^{\alpha} f\big)_{\cH}\big)
\ge\varepsilon_0(\alpha) \big\|A^{\alpha}f\big\|_{\cH} \big\|(A^*)^{\alpha}f\big\|_{\cH},
\quad f\in\dom(A^{\alpha}), \quad \alpha\in (0,1/2).
\end{equation}
$(iii)$ We recall that $\ker (A)=\ker( A^*)$  if $A$ is $m$-accretive, in particular, $\ker(A)$ is a reducing subspace for $A$  (cf., e.g., \cite[p.\ 171]{SF70}). Thus, one can write $A=A_0 \oplus A_1$ with respect to the decomposition $\cH = P_0 \cH \oplus  [I_{\cH} - P_0] \cH$, where $P_0$ denotes the orthogonal projecton onto $\ker(A)$, such that $A_0 = P_0 A P_0 =0$ and $\ker(A_1) =\{0\}$. Thus, also
$A^{\alpha} = A_0 \oplus A_1^{\alpha}$, $\alpha \in (0,1]$, with $\ker(A_1^{\alpha}) = \{0\}$. Hence, one actually obtains
\begin{equation}
\ker (A)=\ker(A^{\alpha})=\ker( A^*), \quad \alpha \in (0,1],    \lb{4.29}
\end{equation}
if $A$ is $m$-accretive.
\end{remark}

\begin{definition}  \lb{d4.6}
Let $A$ be an $m$-sectorial operator in $\cH$ with a vertex $0$ and $B$ a densely
defined, closed operator in $\cH$. Then $B$ is called {\it $A$-form bounded}
(resp., {\it $A$-form compact\,}) if
\begin{equation}
\dom \big(|B|^{1/2}\big)\supseteq \dom \big(A^{1/2}\big) \, \text{ and } \,
|B|^{1/2}(A + I_{\cH})^{-1/2}  \in{\cB}(\cH) \; (\text{resp.,} \in \cB_{\infty}(\cH)).  \lb{4.15}
\end{equation}
\end{definition}

\medskip

Again, $B$ is $A$-form bounded (resp., $A$-form compact) if and only if $|B|$ is.

We also note again that due to the closedness of $|B|^{1/2}$ and $A^{1/2}$,
$\dom \big(|B|^{1/2}\big)\supseteq \dom \big(A^{1/2}\big)$ alone implies that $|B|^{1/2}$ is
$A^{1/2}$-bounded (cf.\ \cite[Remark IV.1.5]{Ka80}), and hence the first condition in \eqref{4.15}
implies the second in connection with form boundedness.

In the following, for simplicity of notation, we agree that for a densely defined linear operator $C$
in $\cH$,
\begin{equation}
\text{the symbol $C^{\#}$ either equals $C$ or $C^*$}.
\end{equation}

\begin{theorem}  \lb{t4.7}
Let $A$ be $m$-sectorial in $\cH$ with a vertex $0$ and assume that
\begin{equation}
\dom \big(A^{1/2}\big) = \dom\big((A^*)^{1/2}\big).
\end{equation}
In addition, suppose that $B$ is a densely defined and closed operator in $\cH$. Then the following assertions hold: \\
$(i)$ $B$ is $A$-form bounded $($resp., $A$-form compact\,$)$ if and only if it
is $A_{\gR}$-form bounded  $($resp., $A_{\gR}$-form compact\,$)$, that is,
\begin{align} \label{equiv1}
& |B|^{1/2}\big(A^{\#} + I_{\cH}\big)^{-1/2} \in \cB(\cH) \; (\text{resp.,} \in  \cB_{\infty}(\cH)) \no \\
& \text{if and only if } \\
&  |B|^{1/2}(A_{\gR} + I_{\cH})^{-1/2} \in \cB(\cH)\; (\text{resp.,} \in \cB_{\infty}(\cH)).  \no
 \end{align}
$(ii)$ The following conditions $(\alpha)$--$(\delta)$ are equivalent:
\begin{align}
\begin{split}
& (\alpha) \quad \big((A^{\#})^*+I_{\cH}\big)^{-1/2}B\big(A^{\#}+I_{\cH}\big)^{-1/2} \,
\text{ is closable in $\cH$,}   \lb{4.18} \\
& (\beta) \quad (A^* + I_{\cH})^{-1/2}B(A^*+I_{\cH})^{-1/2} \,
\text{ is closable in $\cH$,}   \\
& (\gamma) \quad (A + I_{\cH}\big)^{-1/2}B(A+I_{\cH})^{-1/2} \,
\text{ is closable in $\cH$,}   \\
& (\delta) \quad (A_{\gR} + I_{\cH})^{-1/2}B(A_{\gR}+I_{\cH})^{-1/2} \, \text{ is closable in $\cH$.}
\end{split}
\end{align}
$(iii)$ The following conditions $(\alpha)$--$(\delta)$ are equivalent:
\begin{align}
\begin{split}
&  (\alpha) \quad \ol{\big((A^{\#})^*+I_{\cH}\big)^{-1/2}B\big(A^{\#}+I_{\cH}\big)^{-1/2}}
\in \cB(\cH) \; (\text{resp.,} \in \cB_{\infty}(\cH)),    \\
&  (\beta) \quad \ol{(A^*+I_{\cH})^{-1/2}B(A^*+I_{\cH})^{-1/2}} \in \cB(\cH) \; (\text{resp.,}
\in \cB_{\infty}(\cH)),     \label{equiv3}  \\
&  (\gamma) \quad \ol{(A+I_{\cH})^{-1/2}B(A+I_{\cH})^{-1/2}} \in \cB(\cH) \; (\text{resp.,}
\in \cB_{\infty}(\cH)),    \\
& (\delta) \quad \ol{(A_{\gR} + I_{\cH})^{-1/2}B(A_{\gR}+I_{\cH})^{-1/2}} \in  \cB(\cH) \;
(\text{resp.,} \in \cB_{\infty}(\cH)).
\end{split}
\end{align}
\end{theorem}
 \begin{proof}
$(i)$  Since $\dom \big(A^{1/2}\big)=\dom \big(A^{1/2}_{\gR}\big)$ and the operators
$A^{1/2}$ and $A^{1/2}_{\gR}$ are closed, one concludes that the operator $T_{\#}$ defined below,
satisfies
\begin{equation}\label{4.21}
T_{\#}= \big(A^{\#}+I_{\cH}\big)^{1/2}(A_{\gR}+I_{\cH})^{-1/2}\in \cB(\cH) \,  \text{ and } \,
T_{\#}^{-1} = (A_{\gR}+I_{\cH})^{1/2}\big(A^{\#}+I_{\cH}\big)^{-1/2} \in \cB(\cH).
\end{equation}
Therefore, if $|B|^{1/2}(A_{\gR} + I_{\cH})^{-1/2} \in \cB(\cH)$, then also
$|B|^{1/2}\big(A^{\#}+I_{\cH}\big)^{-1/2} \in \cB(\cH)$ and the identity
\begin{equation}
|B|^{1/2}(A_{\gR} + I_{\cH})^{-1/2}T_{\#}^{-1} = |B|^{1/2}\big(A^{\#}+I_{\cH}\big)^{-1/2}
\end{equation}
holds. By \eqref{4.21}, this argument can be reversed, proving the
equivalence \eqref{equiv1}. That $|B|^{1/2}\big(A^{\#} +
I_{\cH}\big)^{-1/2} \in \cB_{\infty}(\cH)$ if and only if
$|B|^{1/2}(A_{\gR} + I_{\cH})^{-1/2} \in \cB_{\infty}(\cH)$ is
proven in the same manner.

$(ii)$ Assume that
$(A_{\gR}+I_{\cH})^{-1/2}B(A_{\gR}+I_{\cH})^{-1/2}$ is closable in
$\cH$. Then so is $\big(T_{\#}^{-1}\big)^*
(A_{\gR}+I_{\cH})^{-1/2}B(A_{\gR}+I_{\cH})^{-1/2} T_{\#}^{-1}$ due
to
\eqref{4.21}. This follows from the following two facts: \\
$(1)$ If $S_1\in\cB(\cH)$ and $S_2$ is a closable (resp., closed) operator in $\cH$, then $S_2S_1$ is closable (resp., closed) in $\cH$. \\
$(2)$ If $T_1$, $T_2$ are closable (resp., closed) operators in $\cH$, and $T_2^{-1}\in\cB(\cH)$, then $T_2T_1$ is closable (resp., closed) in $\cH$ (cf.\ \cite[p.\ 96]{We80}).

Since
\begin{equation}
\big(T_{\#}^{-1}\big)^* (A_{\gR}+I_{\cH})^{-1/2}B(A_{\gR}+I_{\cH})^{-1/2} T_{\#}^{-1}
= \big((A^{\#})^*+I_{\cH}\big)^{-1/2}B\big(A^{\#}+I_{\cH}\big)^{-1/2},     \lb{4.22}
\end{equation}
this proves the closability of
$\big((A^{\#})^*+I_{\cH}\big)^{-1/2}B\big(A^{\#}+I_{\cH}\big)^{-1/2}$
in $\cH$. Again by \eqref{4.21},  this argument can be reversed,
proving the equivalence of $(\alpha)$ and $(\delta)$ in
\eqref{4.18}. The remaining equivalences in \eqref{4.18} follow
from \eqref{4.12} which permits one to individually exchange $A$
and $A^*$ (or $A^*$ and $A$) in the most left and/or most right
factor $(\dots)^{-1/2}$ in $(\alpha)$.

$(iii)$ If $(A_{\gR}+I_{\cH})^{-1/2}B(A_{\gR}+I_{\cH})^{-1/2}$ has
a closure in $\cB(\cH)$ (resp., in $\cB_{\infty}(\cB)$), then \eqref{4.21} and
\eqref{4.22} yield
\begin{align}
\big(T_{\#}^{-1}\big)^* \ol{(A_{\gR}+I_{\cH})^{-1/2}B(A_{\gR}+I_{\cH})^{-1/2}} T_{\#}^{-1}
& = \big(T_{\#}^{-1}\big)^* \ol{(A_{\gR}+I_{\cH})^{-1/2}B(A_{\gR}+I_{\cH})^{-1/2} T_{\#}^{-1}} \no \\
& = \ol{\big(T_{\#}^{-1}\big)^* (A_{\gR}+I_{\cH})^{-1/2}B(A_{\gR}+I_{\cH})^{-1/2} T_{\#}^{-1}}   \no \\
& = \ol{\big((A^{\#})^*+I_{\cH}\big)^{-1/2}B\big(A^{\#}+I_{\cH}\big)^{-1/2}}.    \lb{4.23}
\end{align}
Here we used the following facts: \\
$(1)$ Let $S$ be a bounded operator in $\cH$ with domain $\dom(S)$. Then $S$ is closable and the closure of $S$ has domain $\ol{\dom(S)}\subseteq\cH$. \\
$(2)$ ${\ol S_1} \in \cB(\cH)$, $S_2\in\cB(\cH)$, $\dom(S_1 S_2)$ dense in $\cH$, then
$\ol{S_1 S_2} = {\ol S_1} S_2$. \\
$(3)$ $T_1 \in \cB(\cH)$, ${\ol T_2} \in \cB(\cH)$, then  $\ol{T_1 T_2} = T_1 {\ol T_2}$.

Thus,  $\big((A^{\#})^*+I_{\cH}\big)^{-1/2}B\big(A^{\#}+I_{\cH}\big)^{-1/2}$ has closure in $\cB(\cH)$ (resp., in $\cB_{\infty}(\cB)$). Once more by \eqref{4.21}, this argument is reverseable, proving the equivalence of $(\alpha)$ and $(\delta)$ in \eqref{equiv3}. As in the final part of the proof of item $(ii)$, the remaining equivalences in \eqref{equiv3} follow from \eqref{4.12}.
\end{proof}

To prove one of our  main results  on sectorial operators we next need
a  generalization of Theorem \ref{t3.4} to the sectorial case.

First we recall that $S_1$ is called {\it subordinated} to $S_2$
(cf., e.g., \cite[Sect.\ 14.5]{KZPS76}) if
\begin{equation}
\dom(S_1) \supseteq \dom(S_2), \, \text{ and for some $C>0$, } \,
\|S_1 f\|_{\cH} \le C \|S_2 f\|_{\cH}, \; f \in \dom(S_2).
 \end{equation}

\begin{theorem}  $($\cite[Theorem 1]{Ka61}$)$ \lb{t4.8}
Let $A, B$ be $m$-accretive operators in $\cH$ and assume that $T\in\cB(\cH)$.
In addition, assume that there exists a constant $C>0$ such that
\begin{equation}\label{4.34}
T\dom(A) \subseteq \dom(B) \text{ and } \, \|B Tf\|_{\cH} \le C \|Af\|_{\cH}, \;  f\in\dom(A).
\end{equation}
Then for all $\alpha\in(0,1]$, there exists a constant $C_{\alpha}>0$ such that
\begin{equation}
T\dom(A^{\alpha})\subseteq\dom(B^{\alpha}) \text{ and } \,
\|B^{\alpha}T g\|_{\cH} \le C_{\alpha}\|A^{\alpha}g\|_{\cH}, \; g\in\dom(A^{\alpha}).
\end{equation}
\end{theorem}

In the sequel we need the special case of Theorem \ref{t4.8} corresponding to
$T=I_{\cH}$. However,  it turns out, that this special case is, in fact, equivalent to the general case displayed in Theorem \ref{t4.8}, as will be shown subsequently.

\begin{corollary}   \lb{c4.9}
Suppose $A$ and $B$ are $m$-accretive operators in $\cH$ and $B$
is subordinated to $A$. Then for all $\alpha \in (0,1]$, $B^{\alpha}$ is subordinated to
$A^{\alpha}$,  that is, the inequality
\begin{equation}
\|Bf\|_{\cH} \le C_1\|Af\|_{\cH}, \; f \in \dom(A)\subseteq \dom(B) \lb{3.18a}
\end{equation}
for some constant $C_1>0$ independent of $f\in \dom (A)$, implies
\begin{equation}
\dom (A^\alpha)\subseteq \dom (B^\alpha) \, \text{ and } \,
\|B^\alpha g\|_{\cH} \le
C_{\alpha} \|A^\alpha g\|_{\cH}, \;  g\in\dom (A^\alpha)    \lb{3.19A}
\end{equation}
for some constant $C_\alpha>0$ independent of $g\in \dom (A^\alpha)$.
\end{corollary}

The following result was deduced in \cite{Ka61} from Theorem \ref{t4.8}.  (Actually, it is equivalent to  Theorem \ref{t4.8} as we will show below.) For the sake of completeness we present a short proof based on the generalized polar decomposition \eqref{2.26} and on Corollary \ref{c4.9}.

\begin{theorem}  $($\cite[Theorem 2]{Ka61}$)$ \lb{t4.10}
Let $A$ and $B$ be $m$-accretive operators in ${\cH}$ and let $Q$
be a densely defined closed linear operator in ${\cH}$ such that
$\dom(Q)\supseteq\dom(A), \  \dom(Q^*)\supseteq\dom(B)$ and there exist constants $D_1>0$,
$\wti D_1>0$ such that
\begin{equation} \label{4.31A}
\|Q g\|_{\cH} \le D_1 \|A g\|_{\cH}, \;  g\in\dom(A), \quad \|Q^* f\|_{\cH} \le\wti D_1 \|Bf\|_{\cH}, \;
f\in\dom(B).
\end{equation}
Then for each $\alpha\in(0,1)$, there exists a constant $C_\alpha>0$ such that the following inequality holds:
\begin{equation}
|(f,Qg)_{\cH}|\le C_{\alpha} \|B^{1-\alpha}f\|_{\cH} \|A^{\alpha}g\|_{\cH} , \;
f\in\dom(B), \, g\in\dom(A).
\end{equation}
\end{theorem}
\begin{proof}
By Corollary \ref{c4.9} and the fact that $\|Qg\|_{\cH}=\||Q|g\|_{\cH}$, $\|Q^* f\|_{\cH}=\||Q^*|f\|_{\cH}$, the inequalities \eqref{4.31A}  yield for $\beta, \gamma \in (0,1]$,
\begin{align}
\begin{split}
\big\||Q|^{\beta}g\big\|_{\cH} & \le D_{\beta} \|A^{\beta}g\|_{\cH}, \; g\in\dom(A), \\
\||Q^*|^{\gamma}f\|_{\cH} & \le \wti D_{\gamma}\|B^{\gamma}f\|_{\cH}, \; f\in\dom(B)
\end{split}
\end{align}
for some constants $D_{\beta}>0$, $\wti D_{\gamma}>0$.
On the other hand, by \eqref{2.26}, $Q=|Q^*|^{1-\alpha}U|Q|^{\alpha}$, $\alpha\in [0,1]$.
Combining these facts one arrives at
\begin{align}
\begin{split}
& |(f, Qg)_{\cH}| = |(U^*|Q^*|^{1-\alpha}f, |Q|^{\alpha}g)_{\cH}|
\le\||Q^*|^{1-\alpha}f\|_{\cH} \||Q|^{\alpha} g\|_{\cH} \\
& \quad \le \wti D_{1-\alpha}\|B^{1-\alpha}f\|_{\cH} \, \wti D_{\alpha}\|A^{\alpha}g\|_{\cH},
\;  f\in\dom(B), \, g\in\dom(A),
\end{split}
\end{align}
completing the proof.
\end{proof}

Next we show that Theorem \ref{t4.10}, in fact, implies Theorem \ref{t4.8}. This was stated (without proof) in Kato \cite{Ka61}):

\begin{proof}[Deduction  of Theorem \ref{t4.8} from Theorem \ref{t4.10}] Let $Q=B T$.
Then $Q^*\supseteq T^* B^*$, and
\begin{equation}
\dom(Q^*) \supseteq \dom(B^*) \text{ and } \, \|Q^* f\|\le\|T^*\| \|B^* f\|_{\cH}, \;  f\in\dom(B^*).
\end{equation}
In addition, $T\dom(A)\subseteq\dom(B)$  yields
$\dom(Q) \supseteq \dom(A)$. Therefore, by Theorem  \ref{t4.10} (with $B$ replaced by $B^*$),
for any $\alpha \in (0,1)$,
\begin{equation}
|(f,Qg)_{\cH}|\le C_{\alpha}  \big\|(B^*)^{(1-\alpha)}f\big\|_{\cH} \|A^{\alpha}g\|_{\cH},
\quad f\in\dom(B^*), \,  g\in\dom(A).
\end{equation}
(The case $\alpha=1$ is obvious and needs not be considered.) Hence,
\begin{align}
\begin{split}
|(f,Qg)_{\cH}| = |(f, B T g)_{\cH}|=|((B^*)^{1-\alpha}f, B^{\alpha} T g)_{\cH}|  \le
C_{\alpha} \|(B^*)^{1-\alpha}f\|_{\cH} \|A^{\alpha}g\|_{\cH}, &\\
f\in\dom(B^*), \,  g\in\dom(A).&        \lb{4.50}
\end{split}
\end{align}
Clearly,
\begin{align}
|((B^*)^{1-\alpha}f, B^{\alpha} T g)_{\cH}|
= |(P(B^*)^{1-\alpha}f, B^{\alpha} T g)_{\cH}| =
|(P (B^*)^{1-\alpha}f, PB^{\alpha} T g)_{\cH}|, &\\
f\in\dom(B^*), \,  g\in\dom(A),&
\end{align}
where $P$ is the orthogonal projection onto the closure of
$\ran(B^*)$. Therefore, fixing $g\in\dom(A)$, inequality \eqref{4.50} yields
\begin{equation}
\|PB^{\alpha} T g\|_{\cH} \le C_{\alpha} \|A^{\alpha}g\|_{\cH}, \quad \alpha \in (0,1).
\end{equation}
On the other hand, by \eqref{4.29}, $\ker (B)=\ker(B^{\beta})=\ker( B^*)$, $\beta \in (0,1]$,
since $B$ is $m$-accretive. Therefore,
\begin{equation}
{\ol{\ran (B)}}=\ol{\ran (B^{\beta})}=\ol{\ran (B^*)}, \quad \beta \in (0,1].
\end{equation}
Thus, $P B^{\alpha}T g = B^{\alpha}T g$, $g\in\dom(A)$, and hence finally,
\begin{equation}
\|B^{\alpha} T g\|_{\cH} = \|PB^{\alpha} T g\|_{\cH} \le C_{\alpha} \|A^{\alpha}g\|_{\cH},
\quad g\in\dom(A).
\end{equation}
\end{proof}

Thus we have shown
\begin{equation}
\text{Theorem \ref{4.8}} \; \Longrightarrow \text{ Corollary \ref{4.9}} \; \Longrightarrow
\text{ Theorem \ref{t4.10}} \; \Longrightarrow \text{ Theorem \ref{4.8}}
\end{equation}
and hence the equivalence of Theorem \ref{4.8}, Corollary \ref{4.9}, and 
Theorem \ref{t4.10} (illustrating the usefulness of the generalized polar decomposition \eqref{2.26} in this context).

\medskip

We conclude with the following two results:

\begin{theorem} \lb{t4.11}
Let $A$ be $m$-sectorial in $\cH$ with a  vertex $0$ and assume that $B$ is densely defined
and closed in $\cH$. \\
$(i)$ Suppose that $B$ and $B^*$ are $A_{\gR}$-bounded. Then,
\begin{align}
\begin{split}
& \overline{|B|^{1/2} \big(A^{\#}+I_{\cH}\big)^{-1} |B|^{1/2}}, \,
\overline{|B|^{1/2} \big(A^{\#} + I_{\cH}\big)^{-1} |B^*|^{1/2}} \in \cB(\cH),   \label{4.30} \\
& \overline{|B^*|^{1/2} \big(A^{\#}+I_{\cH}\big)^{-1} |B|^{1/2}}, \,
\overline{|B^*|^{1/2} \big(A^{\#} + I_{\cH}\big)^{-1} |B^*|^{1/2}} \in \cB(\cH).
\end{split}
\end{align}
$(ii)$ Suppose that $B$ and $B^*$ are $A_{\gR}$-bounded and that
$\dom \big(A^{1/2}\big) = \dom \big((A^*)^{1/2}\big)$. Then,
\begin{equation}
|B|^{1/2} \big(A^{\#} + I_{\cH}\big)^{-1/2}, \, |B^*|^{1/2} \big(A^{\#} + I_{\cH}\big)^{-1/2} \in \cB(\cH),
\lb{4.31}
\end{equation}
and
\begin{align}
\begin{split}
& \ol{(A + I_{\cH})^{-1/2} B^{\#} (A + I_{\cH})^{-1/2}}, \,
\ol{(A + I_{\cH})^{-1/2} B^{\#} (A^* + I_{\cH})^{-1/2}}  \in \cB(\cH),   \lb{4.32} \\
& \ol{(A^* + I_{\cH})^{-1/2} B^{\#} (A + I_{\cH})^{-1/2}}, \,
\ol{(A^* + I_{\cH})^{-1/2} B^{\#} (A^* + I_{\cH})^{-1/2}} \in \cB(\cH).
\end{split}
 \end{align}
In particular, $B$ and $B^*$ are  $A^{\#}$-form bounded. Moreover, $B$ and $B^*$ are
$A_{\gR}$-form bounded,
\begin{equation}
|B^{\#}|^{1/2} \big(A_{\gR} + I_{\cH}\big)^{-1/2} \in \cB(\cH),
\lb{4.32a}
\end{equation}
and
\begin{equation}
\overline{(A_{\gR} + I_{\cH})^{-1/2}B^{\#}
(A_{\gR}+I_{\cH})^{-1/2}} \in \cB(\cH).    \lb{4.33}
\end{equation}
$(iii)$  Suppose that $B$ is $A$-bounded and that
$\dom \big(A^{1/2}\big) = \dom \big((A^*)^{1/2}\big)$.
Then $B$ is $A_{\gR}$-form bounded. Moreover, if  $B^*$ is also $A$-bounded, then
equation \eqref{4.33} and the relations \eqref{4.32} hold as well.
\end{theorem}
\begin{proof}
$(i)$  Since $B$ and $B^*$ are $A_{\gR}$-bounded, Theorem \ref{t3.2} implies that
\begin{equation}
|B^{\#}|^{1/2}(A_{\gR} + I_{\cH})^{-1/2}, \,
\overline{(A_{\gR}+I_{\cH})^{-1/2}|B^{\#}|^{1/2}} \in \cB(\cH).   \lb{4.34a}
\end{equation}
Combining these inclusions with \eqref{4.8} one obtains
\begin{align}
& \overline{|B|^{1/2}(A+I_{\cH})^{-1}|B|^{1/2}}   \no \\
& \quad = |B|^{1/2}(A_{\gR}+I_{\cH})^{-1/2}(I_{\cH}+iX)^{-1}
 \overline{(A_{\gR}+I_{\cH})^{-1/2}|B|^{1/2}} \in \cB(\cH),   \lb{4.35} \\
 & \overline{|B|^{1/2}(A^* + I_{\cH})^{-1}|B|^{1/2}}   \no \\
 & \quad = |B|^{1/2}(A_{\gR}+I_{\cH})^{-1/2}(I_{\cH} - iX)^{-1}
 \overline{(A_{\gR}+I_{\cH})^{-1/2}|B|^{1/2}} \in \cB(\cH),    \lb{4.36}
\end{align}
proving the first claim in assertion $(i)$. The remaining three
are proven in precisely the same manner.

$(ii)$ Since by hypothesis
\begin{equation}
\dom\big(A^{1/2}\big) = \dom\big((A^*)^{1/2}\big) = \dom\big(A_{\gR}^{1/2}\big),
\end{equation}
and $A^{1/2}$, $(A^*)^{1/2}$, and $(A_{\gR})^{1/2}$ are closed, one infers that
\begin{equation}
(A_{\gR} + I_{\cH})^{1/2} (A^{\#} + I_{\cH})^{-1/2}, \,
\ol{(A^{\#} + I_{\cH})^{-1/2} (A_{\gR} + I_{\cH})^{1/2}} \in \cB(\cH).   \lb{4.38}
\end{equation}
Hence,
\begin{align}
& |B|^{1/2} \big(A^{\#} + I_{\cH}\big)^{-1/2} = \big[|B|^{1/2} (A_{\gR} + I_{\cH})^{-1/2}\big]
\big[(A_{\gR} + I_{\cH})^{1/2} \big(A^{\#} + I_{\cH}\big)^{-1/2}\big] \in \cB(\cH),   \lb{4.39} \\
& |B^*|^{1/2} \big(A^{\#} + I_{\cH}\big)^{-1/2} = \big[|B^*|^{1/2} (A_{\gR} + I_{\cH})^{-1/2}\big]
\big[(A_{\gR} + I_{\cH})^{1/2} \big(A^{\#} + I_{\cH}\big)^{-1/2}\big] \in \cB(\cH),   \lb{4.40}
\end{align}
applying Theorem \ref{t4.7}\,$(i)$ (also with $B$ replaced by $B^*$).

Using the generalized polar decomposition \eqref{3.1}, $B=|B^*|^{1/2} U |B|^{1/2}$,
one obtains from \eqref{4.34a} and \eqref{4.38} that
\begin{align}
& \ol{(A + I_{\cH})^{-1/2} B (A + I_{\cH})^{-1/2}} \no \\
& \quad = \ol{\big[(A + I_{\cH})^{-1/2} (A_{\gR} + I_{\cH})^{1/2}\big]} \, \ol{(A_{\gR} + I_{\cH})^{-1/2}
B (A_{\gR} + I_{\cH})^{-1/2}} \, \ol{\big[(A_{\gR} + I_{\cH})^{1/2} (A + I_{\cH})^{-1/2}\big]}   \no \\
& \quad = \ol{\big[(A + I_{\cH})^{-1/2} (A_{\gR} + I_{\cH})^{1/2}\big]}
 \ol{\big[(A_{\gR} + I_{\cH})^{-1/2} |B^*|^{1/2}\big]} U \no \\
 & \quad \times \big[|B|^{1/2} (A_{\gR} + I_{\cH})^{-1/2}\big]
 \, \ol{\big[(A_{\gR} + I_{\cH})^{1/2} (A + I_{\cH})^{-1/2}\big]} \in\cB(\cH).   \lb{4.41}
\end{align}
Precisely the same argument works for the remaining three
operators in \eqref{4.32} (using also $B^*=|B|^{1/2} U^*
|B^*|^{1/2}$). Finally, since $A_{\gR} \ge 0$ is self-adjoint, \eqref{4.32a} and \eqref{4.33} follow
from Theorem \ref{t3.2}.

$(iii)$ By Corollary \ref{c4.9},  $|B|^{\alpha}$ is subordinated to
$(A+I_{\cH})^\alpha$, $\alpha\in (0,1]$. In particular, the operator $|B|^{1/2}$ is $(A +
I_{\cH})^{1/2}$-bounded, that is, $|B|^{1/2}(A+I_{\cH})^{-1/2}\in
\cB(\cH)$.\ On the other hand, by \eqref{4.21},  $T =
(A+I_{\cH})^{1/2}(A_{\gR}+I_{\cH})^{-1/2}\in{\cB}(\cH)$. Thus,
\begin{equation} \lb{4.42}
|B|^{1/2}(A_{\gR} + I_{\cH})^{-1/2} = |B|^{1/2}(A + I_{\cH})^{-1/2} (A + I_{\cH})^{1/2}(A_{\gR} + I_{\cH})^{-1/2} = |B|^{1/2}(A + I_{\cH})^{-1/2} T \in \cB(\cH),
\end{equation}
and hence $B$ is $A_{\gR}$-form bounded. If, in addition,  $B^*$ is $A$-bounded, then again by Corollary \ref{c4.9},
$|B^*|^{1/2}(A + I_{\cH})^{-1/2}\in \cB(\cH)$ and hence also  $B^*$ is $A_{\gR}$-form bounded,
\begin{equation}
|B^*|^{1/2}(A_{\gR} + I_{\cH})^{-1/2} = |B^*|^{1/2}(A + I_{\cH})^{-1/2} T \in \cB(\cH).    \lb{4.43}
\end{equation}
Combining \eqref{4.42} and  \eqref{4.43} and using the
generalized polar decomposition \eqref{3.1}, one arrives at
\begin{equation}
\overline{(A_{\gR} + I_{\cH})^{-1/2}B^{\#} (A_{\gR}+I_{\cH})^{-1/2}} = \overline{(A_{\gR} + I_{\cH})^{-1/2}
    |(B^{\#})^*|^{1/2}} U |B^{\#}|^{1/2}(A_{\gR}+I_{\cH})^{-1/2} \in \cB(\cH).    \lb{4.44}
\end{equation}
Relations \eqref{4.32} then follow as in the proof of item $(ii)$.
\end{proof}

Finally, we state an analog of Theorem \ref{t4.11} in connection with relative (form) compactness:

\begin{theorem} \lb{t4.12}
Let $A$ be $m$-sectorial in $\cH$ with a  vertex $0$, assume that $B$ is densely
defined and closed in $\cH$. \\
$(i)$Suppose that $\dom(B)\cap\dom(B^*)\supseteq\dom(A_{\gR})$ and that $B$
$($resp., $B^*$$)$ is $A_{\gR}$-compact. Then,
\begin{align}
& \overline{|B|^{1/2} \big(A^{\#}+I_{\cH}\big)^{-1} |B|^{1/2}}, \,
\overline{|B|^{1/2} \big(A^{\#} + I_{\cH}\big)^{-1} |B^*|^{1/2}}, \,
\overline{|B^*|^{1/2} \big(A^{\#}+I_{\cH}\big)^{-1} |B|^{1/2}}\in \cB_\infty(\cH),  \no \\
& \big(\text{resp., } \overline{|B|^{1/2} \big(A^{\#} + I_{\cH}\big)^{-1} |B^*|^{1/2}}, \,
\overline{|B^*|^{1/2} \big(A^{\#}+I_{\cH}\big)^{-1} |B|^{1/2}},  \label{4.60}  \\
& \hspace*{1.16cm}
\overline{|B^*|^{1/2} \big(A^{\#} + I_{\cH}\big)^{-1} |B^*|^{1/2}} \in \cB_\infty(\cH).\big)   \no
\end{align}
$(ii)$ Suppose that $\dom(B)\cap\dom(B^*)\supseteq\dom(A_{\gR})$ and that  $B$
$($resp., $B^*$$)$ is $A_{\gR}$-compact. In addition, assume that
$\dom \big(A^{1/2}\big) = \dom \big((A^*)^{1/2}\big)$. Then,
\begin{equation}
|B|^{1/2}(A^{\#} + I_{\cH})^{-1/2}\in \cB_\infty(\cH) \;
\big(\text{resp., } |B^*|^{1/2}\big(A^{\#} + I_{\cH}\big)^{-1/2} \in \cB_\infty(\cH)\big),  \label{4.61}
\end{equation}
and
\begin{align}
\begin{split}
& \ol{(A + I_{\cH})^{-1/2} B^{\#} (A + I_{\cH})^{-1/2}}, \,
\ol{(A + I_{\cH})^{-1/2} B^{\#} (A^* + I_{\cH})^{-1/2}} \in \cB_\infty(\cH),   \lb{4.62} \\
& \ol{(A^* + I_{\cH})^{-1/2} B^{\#} (A + I_{\cH})^{-1/2}}, \,
\ol{(A^* + I_{\cH})^{-1/2} B^{\#} (A^* + I_{\cH})^{-1/2}} \in \cB_\infty(\cH).
\end{split}
\end{align}
In particular, $B$ $($resp., $B^*$$)$ is $A^{\#}$-form compact. Moreover, $B$ $($resp., $B^*$$)$ is
$A_{\gR}$-form compact,
\begin{equation}
|B|^{1/2} \big(A_{\gR} + I_{\cH}\big)^{-1/2} \in \cB_{\infty}(\cH) \;  \big(\text {resp., }
|B^*|^{1/2}(A_{\gR}+I_{\cH})^{-1/2} \in\cB_\infty (\cH)\big)     \lb{4.63}
\end{equation}
and
\begin{equation}
\overline{(A_{\gR} + I_{\cH})^{-1/2}B^{\#}
(A_{\gR}+I_{\cH})^{-1/2}} \in \cB_{\infty}(\cH).    \lb{4.64}
\end{equation}
$(iii)$  Suppose that $\dom(B)\cap\dom(B^*)\supseteq\dom(A)$ and that $B$
$($resp., $B^*$$)$ is $A^{1-\varepsilon}$-compact for some $\varepsilon \in (0,1)$. In addition, assume that $\dom \big(A^{1/2}\big) = \dom \big((A^*)^{1/2}\big)$.
Then $B$ $($resp., $B^*$$)$ is $A_{\gR}$-form compact. Moreover, equation
\eqref{4.64} and relations \eqref{4.62} hold as well.
\end{theorem}
\begin{proof}
$(i)$  Since by hypothesis $B$ and $B^*$ are $A_{\gR}$-bounded and $B$ (resp., $B^*$) is
$A_{\gR}$-compact, Theorem \ref{t3.5} implies that
\begin{align}
\begin{split}
& |B|^{1/2}(A_{\gR} + I_{\cH})^{-1/2} =
\overline{(A_{\gR}+I_{\cH})^{-1/2}|B|^{1/2}} \in \cB_\infty(\cH)    \lb{4.70} \\
& \big(\text{resp., } \, |B^*|^{1/2}(A_{\gR} + I_{\cH})^{-1/2} =
\overline{(A_{\gR}+I_{\cH})^{-1/2}|B^*|^{1/2}} \in \cB_\infty(\cH)\big).
\end{split}
\end{align}
At this point one can follow the proof of Theorem \ref{t4.11}\,$(i)$, noting that each operator in
\eqref{4.35} and \eqref{4.36} contains at least one compact factor from \eqref{4.70}.

$(ii)$ Again, one can follow the proof of Theorem \ref{t4.11}\,$(ii)$, noting that the right-hand side of
\eqref{4.39} (resp., \eqref{4.40}) contains a compact factor from \eqref{4.70}. Similarly, the right-hand side of \eqref{4.41} and the analogous equations with $A$ replaced by $A^*$ (resp., $B$ replaced by $B^*$) contains at least one compact factor from \eqref{4.70}. Relations \eqref{4.63} and
\eqref{4.64} are clear from Theorem \ref{t3.5} since $A_{\gR} \ge 0$ is self-adjoint.

$(iii)$ Since by hypothesis, $\dom(B)\cap\dom(B^*)\supseteq\dom(A)$, $B$ and $B^*$ are $A$-bounded and hence Theorem \ref{t4.11}\,$(iii)$ and the results \eqref{4.42}--\eqref{4.44} in its proof are at our disposal. Next, we first assume that
$B(A+I_{\cH})^{-1+\varepsilon}\in \cB_{\infty}(\cH)$. Then (using  $|B|=U^* B$, cf.\ \eqref{2.22}),
\begin{equation}
|B|(A+I_{\cH})^{-1+\varepsilon_0+i\gamma} =\big[|B|(A+I_{\cH})^{-1+\varepsilon}\big]
(A+I_{\cH})^{-(\varepsilon-\varepsilon_0) + i \gamma} \in \cB_{\infty}(\cH),   
\quad 0 \leq \varepsilon_0 < \varepsilon  \lb{4.71}
\end{equation}
since
\begin{equation}
(A+I_{\cH})^{-\beta + i \gamma} \in\cB(\cH), \quad \beta \in (0,1), \;
\gamma \in \bbR,    \lb{4.72}
\end{equation}
as is clear from the formula (cf.\ \cite[Remark V.3.50]{Ka80}, 
\cite[Sect.\ 14.12]{KZPS76}),
\begin{equation}
(S+I_{\cH})^{-z} = \f{\sin(\pi z)}{\pi} \int_0^\infty dt \, t^{-z} (S+(t+1)I_{\cH})^{-1}, \quad
z\in\bbC, \; \Re(z) \in (0,1),   \lb{4.73}
\end{equation}
for any $m$-accretive operator $S$ in $\cH$.

Since by hypothesis
$B(A+I_{\cH})^{-1+\varepsilon} \in \cB_{\infty}(\cH) \subset \cB(\cH)$, $B$ is
subordinated to $(A+I_{\cH})^{1-\varepsilon}$, and hence by Corollary \ref{c4.9},
$|B|^\alpha$ is subordinated to $(A+I_{\cH})^{(1-\varepsilon)\alpha}$ for all
$\alpha\in (0,1]$,
\begin{equation}
|B|^{\alpha} (A+I_{\cH})^{-(1-\varepsilon)\alpha} \in \cB(\cH), \quad \alpha \in (0,1].
\lb{4.74}
\end{equation}

In the following we assume without loss of generality that
\begin{equation}
\ker(|B|) = \ker (B) = \{0\}. 
\end{equation}

Thus, one obtains 
\begin{align}
&\ol{(A^*+I_{\cH})^{-z} |B| (A+I_{\cH})^{-1+z}} =
\ol{(A^*+I_{\cH})^{-z} |B|^z |B|^{1-z} (A+I_{\cH})^{-1+z}}   \no \\
& \quad = \big[|B|^{{\ol z}}(A+I_{\cH})^{-{\ol z}}\big]^* \big[|B|^{1-z}(A+I_{\cH})^{-1+z}\big]
\in \cB(\cH), \quad \Re(z)\in (0,1),     \lb{4.75}
\end{align}
since by \eqref{4.72} and \eqref{4.74},
\begin{align}
\begin{split}
|B|^{\alpha + i \beta} (A+I_{\cH})^{-\alpha - i \beta}
= |B|^{i\beta} \big[|B|^{\alpha} (A+I_{\cH})^{-(1-\varepsilon)\alpha}\big]
(A+I_{\cH})^{-\varepsilon \alpha - i \beta} \in \cB(\cH),&   \lb{4.76} \\
 \alpha \in (0,1], \; \beta\in\bbR,&
 \end{split}
\end{align}
as $|B|^{i \beta}$ is unitary. Moreover, choosing a compact subinterval of $(0,1)$ containing $1/2$ in its interior, for instance, $[\varepsilon_0,1-\varepsilon_0]$ for some
$\varepsilon_0 \in (0,1/2)$, one obtains for
$z=\varepsilon_0+i\gamma$ in \eqref{4.75},
\begin{align}
& \|\ol{(A^*+I_{\cH})^{-\varepsilon_0-i\gamma}|B|
(A+I_{\cH})^{-(1-\varepsilon_0)+i\gamma}}\|  \no \\
& \quad = \big\|\big[|B|^{\varepsilon_0-i\gamma}
(A+I_{\cH})^{-\varepsilon_0+i\gamma}\big]^*
\big[|B|^{1-\varepsilon_0 -i\gamma}
(A+I_{\cH})^{-(1-\varepsilon_0)+i\gamma}\big]\big\|   \no \\
& \quad \leq \big\||B|^{\varepsilon_0} (A+I_{\cH})^{-\varepsilon_0+i\gamma}\big\|
\big\||B|^{1-\varepsilon_0} (A+I_{\cH})^{-(1-\varepsilon_0)+i\gamma}\big\|    \no \\
& \quad \leq \big\||B|^{\varepsilon_0} (A+I_{\cH})^{-(1-\varepsilon)\varepsilon_0}\big\|
\big\|(A+I_{\cH})^{-\varepsilon \varepsilon_0+i\gamma}\big\|  \no \\
& \quad \quad \times \big\||B|^{1-\varepsilon_0}
(A+I_{\cH})^{-(1-\varepsilon)(1-\varepsilon_0)}\big\|
\big\|(A+I_{\cH})^{-\varepsilon(1-\varepsilon_0)+i\gamma}\big\|   \no \\
& \quad \leq \big\||B|^{\varepsilon_0} (A+I_{\cH})^{-(1-\varepsilon)\varepsilon_0}\big\|
\big\||B|^{1-\varepsilon_0} (A+I_{\cH})^{-(1-\varepsilon)(1-\varepsilon_0)}\big\| 
C e^{2 \pi |\gamma|}
\lb{4.76a}
\end{align}
for some $C=C(\varepsilon,\varepsilon_0)>0$ (cf.\ \eqref{4.76b}). (In fact, using \cite[Theorem 4]{Ka62}, one can replace $2\pi$ by $\pi$ in the exponent of \eqref{4.76a}, but this plays no role in our context.) Here we used the fact that by \eqref{4.1} and \eqref{4.73},
\begin{align}
\big\|(A+I_{\cH})^{-z}\big\| &= \bigg|\f{\sin(\pi z)}{\pi}\bigg| \,
\bigg\|\int_0^\infty dt \, t^{-z} (A+(t+1)I_{\cH})^{-1}\bigg\|   \no \\
& \leq \bigg|\f{\sin(\pi z)}{\pi}\bigg| \int_0^\infty dt \, t^{-\Re(z)}
\big\|(A+(t+1)I_{\cH})^{-1}\big\|   \no \\
& \leq \bigg|\f{\sin(\pi z)}{\pi}\bigg| \int_0^\infty dt \,\f{ t^{-\Re(z)}}{t+1}
 = \bigg|\f{\sin(\pi z)}{\sin(\pi \Re(z))}\bigg|, \quad \Re(z) \in (0,1).   \lb{4.76b}
\end{align}
The same computation applies to $z=1-\varepsilon_0+i\gamma$ in \eqref{4.75}, and more generally, one has
\begin{equation}
\sup_{\gamma\in\bbR} 
\big\|\ol{(A^*+I_{\cH})^{-\alpha-i \gamma} |B| (A+I_{\cH})^{-1+\alpha + i\gamma}}\big\|
e^{-2\pi|\gamma|} < \infty, \quad \alpha \in (0,1).    \lb{4.76c}
\end{equation}

In addition, the map
\begin{equation}
z\mapsto e^{z^2} \ol{(A^*+I_{\cH})^{-z} |B| (A+I_{\cH})^{-1+z}} \, 
\text{ is analytic in the strip
$\Re(z) \in (0,1)$}.   \lb{4.77}
\end{equation}
By the proof of the Lemma in \cite[p.\ 115]{RS78}, \eqref{4.71},
\eqref{4.75}, \eqref{4.76c} (for $\alpha = \varepsilon_0$ and
$\alpha = 1- \varepsilon_0$), and \eqref{4.77} imply,  by complex interpolation, that
\begin{equation}
e^{z^2} \ol{(A^*+I_{\cH})^{-z} |B| (A+I_{\cH})^{-1+z}} \in \cB_{\infty}(\cH), 
\quad z\in\bbC, \; \Re(z)\in (\varepsilon_0,1-\varepsilon_0).   \lb{4.78}
\end{equation}
Since $\varepsilon_0\in(0,1/2)$ can be taken arbitrarily small, one finally concludes that
\begin{equation}
\ol{(A^*+I_{\cH})^{-z} |B| (A+I_{\cH})^{-1+z}} \in \cB_{\infty}(\cH), \quad z\in\bbC, \;
\Re(z)\in (0,1).   \lb{4.78a}
\end{equation}
In particular,
\begin{equation}
(A^*+I_{\cH})^{-1/2} |B| (A+I_{\cH})^{-1/2} \in \cB_{\infty}(\cH).    \lb{4.79}
\end{equation}
Thus,
\begin{align}
& \ol{(A_{\gR}+I_{\cH})^{-1/2} |B| (A_{\gR}+I_{\cH})^{-1/2}}  \no \\
& \quad = \ol{(A_{\gR}+I_{\cH})^{-1/2} (A^*+I_{\cH})^{1/2} (A^*+I_{\cH})^{-1/2}
|B| (A+I_{\cH})^{-1/2} (A+I_{\cH})^{1/2} (A_{\gR}+I_{\cH})^{-1/2}}  \no \\
& \quad = \big[(A+I_{\cH})^{1/2} (A_{\gR}+I_{\cH})^{-1/2}\big]^*
\ol{(A^*+I_{\cH})^{-1/2} |B| (A+I_{\cH})^{-1/2}}
\big[(A+I_{\cH})^{1/2} (A_{\gR}+I_{\cH})^{-1/2}\big]  \no \\
& \quad = T^* \, \ol{(A^*+I_{\cH})^{-1/2} |B| (A+I_{\cH})^{-1/2}} \, T \in \cB_{\infty}(\cH),
\lb{4.80}
\end{align}
where $T=[(A+I_{\cH})^{1/2} (A_{\gR}+I_{\cH})^{-1/2} \in \cB(\cH)$ (cf.\ \eqref{4.21})
and we used again the reasoning $(1)$--$(3)$ as in the proof of \eqref{4.23}. Relation \eqref{4.80} and the fact that an operator $D$ is compact if and only if $D^*D$ is, then finally implies
\begin{equation}
|B|^{1/2} (A_{\gR}+I_{\cH})^{-1/2} \in \cB_{\infty}(\cH).   \lb{4.81}
\end{equation}
In exactly the same manner, the assumption
$B^* (A+I_{\cH})^{-1+\varepsilon} \in \cB_{\infty}(\cH)$ then implies
\begin{equation}
|B^*|^{1/2} (A_{\gR}+I_{\cH})^{-1/2} \in \cB_{\infty}(\cH).   \lb{4.82}
\end{equation}
In particular, since the operator in \eqref{4.42} (resp. in \eqref{4.43}) now lies in
$\cB_{\infty}(\cH)$, $B$ (resp., $B^*$) is $A_{\gR}$-form compact, that is, \eqref{4.63} holds.
Equation \eqref{4.64} then follows as in \eqref{4.44} from \eqref{4.63}. Finally, relations \eqref{4.62}
again follow as in the proof of item $(ii)$.
\end{proof}

\begin{remark} \lb{r4.13}
We do not know if one can generally take $\varepsilon = 0$ in
Theorem \ref{t4.12}\,$(iii)$.
Of course, if the condition
\begin{equation}
\sup_{\gamma\in\bbR}\big\|(A+I_{\cH})^{i \gamma}\big\|<\infty    \lb{4.98}
\end{equation}
holds, the proof of Theorem \ref{t4.12}\,$(iii)$ (c.f., in particular, estimates \eqref{4.76a}) shows that $\varepsilon$ can indeed be taken equal to zero. In particular, \eqref{4.98} holds if $A$ is similar to a self-adjoint operator $S$ in some complex, separable Hilbert space $\cH'$ with
$S\geq -I_{\cH'}$ and $\{-1\}$ not an eigenvalue of $S$ (by applying the spectral theorem to $S$). Conversely, suppose $A$ is $m$-sectorial in $\cH$ with a  vertex $0$ and consider $T=(A+I_{\cH})^{-i}=\bigl((A + I_{\cH})^{-1}\bigr)^i$. Then by \eqref{4.98}, $T^t$, $t\in\bbR$, is a uniformly bounded one-parameter commutative group of transformations, in fact, a $C_0$-group with generator $i \log\big((A+I_{\cH})^{-1}\big)$ (cf.\ the discussion in \cite[Corollary 5.4]{Ok00}), 
\begin{equation}
T^t = (A+I_{\cH})^{-it} = e^{it \log((A+I_{\cH})^{-1})},  \quad 
\big\|T^t\big\| \leq C, \quad  t\in\bbR,       \lb{4.99} 
\end{equation}
for some fixed constant $C>0$. Thus, by  Sz.-Nagy's theorem \cite{Sz47} (see also \cite[Sect.\ I.6]{DK74}, \cite[Lemma XV.6.1]{DS88a}), there exists an operator $V\in\cB(\cH)$ with $V^{-1}\in\cB(\cH)$, such that 
\begin{equation}
V^{-1} T^t V = U(t) = e^{it H}, \quad t\in\bbR,   \lb{4.100}
\end{equation}
where $U(t)$, $t\in\bbR$, is a strongly continuous unitary one-parameter group with a self-adjoint (possibly unbounded) generator $H=H^*$ in $\cH$. Thus,
\begin{equation}
T^t=e^{it \log((A+I_{\cH})^{-1})}=Ve^{it H}V^{-1} = e^{it VHV^{-1}}, \quad t\in\bbR,  
\lb{4.101}
\end{equation}
implying
\begin{equation}
\log\big((A+I_{\cH})^{-1}\big) = V H V^{-1}.    \lb{4.102}
\end{equation}
On the other hand (cf.\ \cite[Proposition 2.1]{Ok00}), $\log((A+I_{\cH})^{-1})$ is also the generator of a $C_0$-semigroup of contractions in $\cH$,
\begin{equation}
(A+I_{\cH})^{-t} = e^{t \log((A+I_{\cH})^{-1})}, \quad t\geq 0,   \lb{4.103}
\end{equation}
and hence,
\begin{equation}
(A+I_{\cH})^{-t} = e^{t \log((A+I_{\cH})^{-1})} = e^{t V H V^{-1}} 
= V e^{t H} V^{-1}, \quad t\geq 0.   \lb{4.104}
\end{equation}
Taking $t=1$ in \eqref{4.104} then shows that $A$ is similar to a self-adjoint operator 
in $\cH$. (Incidentally, we note that necessarily $H \leq  c I_{\cH}$ for some $c\in\bbR$,  since \eqref{4.103} represents a family of contractions.)
\end{remark}

\bigskip

\noindent {\bf Acknowledgments.}
We are indebted to Brian Davies, Nigel Kalton, Heinz Langer, Yuri Latushkin, Vladimir  Ovchinnikov, Leiba Rodman, and Barry Simon for helpful correspondence.

One of us (F.G.) gratefully acknowledges the extraordinary hospitality of the Faculty
of Mathematics of the University of Vienna, Austria, and especially, that of Gerald Teschl,
during his three month visit in the first half of 2008, where parts of this paper were written.


\end{document}